\documentclass[11pt]{amsart}
\usepackage{amssymb}
\usepackage{amsfonts}
\usepackage{amsmath}
\usepackage{graphicx}
\usepackage{xcolor}
\usepackage{amsmath}
\usepackage{mathrsfs}
\usepackage{stmaryrd}
\usepackage{epsfig,color}
\usepackage{blindtext}
\usepackage{enumerate}
\usepackage{hyperref}
\usepackage{xcolor}
\usepackage{url}
\usepackage{bbm}
\usepackage{nicefrac,mathtools}
\usepackage{bm}   
\usepackage{tabularx}
\DeclareGraphicsExtensions{.pdf,.jpeg,.png}
\usepackage{epstopdf}
\usepackage{cancel} 
\usepackage[normalem]{ulem} 
\usepackage{verbatim} 
\usepackage{scalerel}
\usepackage{enumitem}
\usepackage[all, cmtip]{xy}

\usepackage{color}
\usepackage[msc-links, lite]{amsrefs}
\usepackage{geometry}
\usepackage{stackengine,scalerel}

\usepackage{tikz-cd}

\newcommand\ttimes{\mathbin{\ThisStyle{\ensurestackMath{%
  \stackengine{-1\LMpt}{\SavedStyle\times}
  {\SavedStyle_{\hstretch{.9}{\mkern1mu\sim}}}{O}{c}{F}{T}{S}}}}}

\newtheorem{theorem}{Theorem}[section]

\newtheorem{proposition}[theorem]{Proposition}
\newtheorem{lemma}[theorem]{Lemma}

\newtheorem{claim}[]{Claim}
\newtheorem{fact}[]{Fact}

\theoremstyle{definition}
\newtheorem{definition}[theorem]{Definition}
\theoremstyle{remark}

\numberwithin{equation}{section}

\newcommand{\mf}{\mathbf}
\newcommand{\mb}{\mathbb}
\newcommand{\mc}{\mathcal}

\newcommand{\mk}{\mathfrak}
\newcommand{\mr}{\mathrm}
\newcommand{\oli}{\overline}
\newcommand{\omc}[1]{\oli{\mc #1}}
\newcommand{\eps}{\varepsilon}
\newcommand{\wti}{\widetilde}

\newcommand{\smf}{\smallfrown}
\newcommand{\Is}{\mathfrak{Is}}

\newcommand{\Area}{\mathrm{Area}}
\newcommand{\Id}{\mathrm{Id}}

\newcommand{\rom}[1]{\expandafter\romannumeral #1}
\newcommand{\Rom}[1]{\uppercase\expandafter{\romannumeral #1}}

\DeclareMathOperator{\Index}{index}

\DeclareMathOperator{\Diff}{Diff}
\DeclareMathOperator{\spt}{spt}

\DeclareMathOperator{\image}{image}

\title[Embedded minimal tori in three-spheres]{Existence of embedded minimal tori in three-spheres with positive Ricci curvature}

\author{Xingzhe Li}
\address{Cornell University, Department of Mathematics, Ithaca, New York 14850}
\email{xl833@cornell.edu}

\author{Zhichao Wang}
\address{Shanghai Center for Mathematical Science, 2005 Songhu Road, Fudan University, Shanghai, 200438, China}
\email{zhichao@fudan.edu.cn}

\begin{document}

\begin{abstract}
In this paper, we prove the strong Morse inequalities for the area functional in the space of embedded tori and spheres in the three sphere. As a consequence, we prove that in the three dimensional sphere with positive Ricci curvature, there exist  at least 4 distinct embedded minimal tori. Suppose in addition that the metric is bumpy, then the three-sphere contains at
least 9 distinct embedded minimal tori. The proof relies on a multiplicity one theorem for the Simon-Smith min-max theory proved by the second author and X. Zhou \cite{wangzc2023existenceFour}. 
\end{abstract}

\maketitle



\section{Introduction}
In his famous 1982 Problem Section \cite{Yau82}, S. T. Yau posed the problem to prove the existence of four distinct embedded minimal spheres in any manifold diffeomorphic to $S^3$. Around the same time, Simon-Smith \cite{Smith82} proved the existence of at least one embedded minimal sphere in $S^3$ with an arbitrary metric, using a variant of the min-max theory for minimal hypersurfaces developed by Almgren \cites{Alm62,Alm65} and Pitts \cite{Pi}; see also Schoen-Simon \cite{SS} and Colding-De Lellis \cite{Colding-DeLellis03}. Later on, White \cite{Whi91}, Haslhofer-Ketover \cite{HK19} proved the existence of at least two embedded minimal spheres for metrics with positive Ricci or bumpy metrics. Recently, the second author together with X. Zhou  confirmed this conjecture for bumpy metrics and metrics with positive Ricci curvature by developing the Multiplicity One Theorem for Simon-Smith min-max theory; see \cite{wangzc2023existenceFour}.

Motivated by the progress on Yau's conjecture, we set out to investigate the high genus problem. It is noteworthy that the space of all unoriented Clifford tori in round $S^3$ can be parametrized by $\mathbb{RP}^2 \times \mathbb{RP}^2$, whose Lusternik-Schnirelmann number and Morse number are equal to $5$ and $9$, respectively. By a perturbation argument applied to the area functional \cite{Whi91}*{Theorem 3.2}, White proved that every sufficiently round metric on $S^3$ admits at least five distinct minimal embedded tori, and a generic such metric must admit at least nine distinct minimal embedded tori. One would then expect five distinct minimal tori in a Riemannian $S^3$. In this work, by employing the multiplicity one theorem for Simon-Smith min-max theory \cite{wangzc2023existenceFour}, we study the existence of embedded minimal tori in $S^3$ with positive Ricci metrics and prove our main theorem as follows.  

\begin{theorem}\label{thm: main theorem}
Assume that $g$ is a metric positive Ricci curvature on $S^3$. Then there exist at least four distinct embedded minimal tori in $(S^3, g)$. Suppose in addition that $g$ is bumpy. Then there exist at least nine distinct embedded minimal tori. In particular, they have index 5,6,6,7,7,7,8,8,9. 
\end{theorem}

Our Theorem \ref{thm: main theorem} is significant in the sense that it gives an affirmative answer to the existence of multiple minimal tori in any positively Ricci curved $S^3$. We point out that by generalizing the multiplicity one theorem for Simon-Smith min-max theory to $G$-invariant case, the first author together with T. Wang and X. Yao \cite{XZ-TR-YX} construct the minimal tori in lens space which has Heegaard genus one.

The main technique for constructing multiple minimal tori is the relative strong Morse inequalities for the area functional on the space of surfaces with bounded topology. Analogous inequalities on the space of mod $2$ flat chains were first confirmed by Marques-Montezuma-Neves in \cite{Marques-Montezuma-Neves20}, based on which they established a Morse theory for minimal hypersurfaces in a Riemannian manifold $(M^{n+1}, g)$ of dimension $3 \leq (n + 1) \leq 7$. The essence is that generically, a closed minimal hypersurface produced by min-max over $k$-parameter sweepouts must have index equal to $k$ \cites{Marques-Neves16, MN18}. The proof shares the same spirit as that of \cite{MN18}*{Theorem 7.2} since both proofs rely on the local min-max property of White \cite{Whi94}. Notice that the strong Morse inequalities in \cite{Marques-Montezuma-Neves20} directly imply the existence of infinitely many minimal hypersurfaces for generic case; see also \cite{Zhou19} and a complete proof to this conjecture by A. Song \cite{Song18}.

Once we move into the realm of minimal surfaces of controlled topological types, the strong Morse inequalities in \cite{Marques-Montezuma-Neves20} become insufficient to tell the topology apart. As one key novelty of this paper, we propose and prove the relative strong Morse inequalities for the area functional on the space of surfaces with bounded topology in positively Ricci curved $S^3$, which help identify minimal tori out of minimal spheres. In the following, denote by $c_k(\ell)$ the number of minimal tori with index equal to $k$ and area less than $\ell$, and by $\beta_k(\ell)$ the rank of $H_k(\oli{\mc T}^\ell,\oli{\mc S}^\ell;\mb Z_2)$, where $\oli{\mc T}^\ell$ (resp. $\oli{\mc S}^\ell$) represent the space of embedded tori and spheres (resp. embedded spheres). 
 
\begin{theorem}[Relative Strong Morse Inequalities]\label{thm: main theorem 2}
 For any fixed $\ell\in (0,\infty)$, we have $\beta_k(\ell)<\infty$ and the Strong Relative Morse Inequalities for every $k\in \mb Z_+$:
 \[    c_k(\ell)-c_{k-1}(\ell)+\cdots +(-1)^{k}c_0(\ell)\geq \beta_k(\ell)-\beta_{k-1}(\ell)+\cdots+(-1)^k\beta_0(\ell).
 \]
 In particular, we have 
 \[    c_k(\ell)\geq \beta_k(\ell)\]
for every $k\in \mb Z_+$.
\end{theorem}

The Morse theory for the area functional has been built on a series of paper leaded by Marques and Neves. In \cite{Marques-Neves16}, they proved the index upper bound, which also holds true for the Simon-Smith setting. Later on,  X. Zhou solved the Multiplicity One Conjecture in an exciting work \cite{Zhou19} based on his joint work \cites{ZZ18, zz17} with Z. Zhu; see \cite{CM20} for three dimensional case by Chodosh-Mantoulidis. Using X. Zhou's breakthrough, Marques-Neves proved the index lower bound for the volume spectrum (or in the homology classes) in their celebrated work \cite{MN18}. Then in a recent work \cite{Marques-Montezuma-Neves20} by Marques-Montezuma-Neves, they proved the strong Morse inequalities, thus completing the Morse-theoretic program. In this work, we build a Morse theory on the space of tori in $S^3$, which also holds true for any closed three manifold with positive Ricci curvature. Moreover, we develop a relative version to avoid the genus degeneration, thus allowing us to construct embedded minimal tori. 

\subsection{Sketch of the proof}

To prove Theorem \ref{thm: main theorem 2}, we follow similar strategy as the proof of the strong Morse inequalities on the space of mod $2$ flat chains \cite{Marques-Montezuma-Neves20}, that is, to set up the min-max theory in a nontrivial $k$-th homology class of a level set, and then employ an interpolation (cut-and-paste procedure) to show that the local min-max family of an index $k$ minimal surface inside the level set uniquely and only contributes to the $k$-th homology class. 

However, the Simon-Smith min-max theory only has genus upper bound which may produce minimal spheres by a sweepout of tori even in the multiplicity one case. To overcome it, we study the relative homology min-max in the space $(\omc T,\omc S)$ (see Section \ref{sec: preliminaries} for definitions). Then the major problem is to determine the homology groups of such relative spaces, which is highly nontrivial to our knowledge.

\medskip
Let us briefly introduce the new ideas to address the above two issues. Our proof of genus lower bound is inspired by Marques-Neves' proof of the Morse index lower bound \cite{MN18}. We argue by contradiction. By doing another tightening process, we may assume that in the minimizing sequence, the surfaces with area close to the min-max value should be close to a smooth sphere. By applying the two homotopies in \cite{MN18}, one may assume that the surfaces of large area are intrinsically generated by the $k$-parameter family (see \cite{MN18}*{Section 6.1} for White's Local Min-max Theorem) that decreases the area of minimal spheres; the Multiplicity One Theorem for Simon-Smith min-max theory proved by the second author and X. Zhou in \cite{wangzc2023existenceFour} is essentially used in this part. Then we cut off the small chains that lie in the image of the homology class in the relative spaces. Since the surfaces that we cut off are all spheres, it will not affect the homology class (clearly, they are not in the same homotopy class). Moreover, the area of maximal slices is deformed down by a definite amount during this process, which contradicts the definition of min-max width. This proves the genus lower bound. 

To determine the topology of the relative spaces, we consider it on the round three sphere. The three main ingredients are the Smale Conjecture by A. Hatcher \cite{Hat83}, the Perturbation Theorem by B. White \cite{Whi91}*{Theorem 3.2} and the Lawson Conjecture by S. Brendle \cite{Brendle-Lawson-conj}. Since each embedded minimal torus has area equal to $2\pi^2$ (by Lawson Conjecture), it follows from our relative homology min-max theorem that $\omc T$ is ``weakly homotopic" to $\omc T^{2\pi^2+\eps}\cup \omc S$, the space of tori with area bounded by $2\pi^2+\eps$ union the space of spheres. Similarly, $\omc T^{2\pi^2-\eps}\cup \omc S$ is ``weakly homotopic" to $\omc S$. Hence the problem has been reduced to determine the homology classes of $(\omc T^{2\pi^2+\eps}\cup \omc S,\omc T^{2\pi^2-\eps}\cup \omc S)$. We will construct a 9-parameter family of tori and prove that it induces an injective map between the relative homology groups. This part involves the Smale Conjecture and the argument by Marques-Montezuma-Neves \cite{Marques-Montezuma-Neves20}*{Claim 3.7}. Finally, we use White's Perturbation Theorem to show that there are no other relative homology classes.

\subsection{Outline of the paper}
The paper is organized as follows. In Section \ref{sec: preliminaries}, we collect preliminary materials, and prove the interpolation theorem and deformation theorem. In Section \ref{sec: relative homology min-max theory}, we establish the index lower bound for Simon-Smith homology min-max theory and show the relative homology min-max theorem. This sets up the stage for proving the relative strong Morse inequalities in Section \ref{sec: morse inequalities for relative spaces}. We compute the homology groups of the relative spaces and finish the proof of our main theorem in Section \ref{sec: construction of relative sweepouts}.     

\subsection*{Acknowledgement}
We would like to thank Professor Xin Zhou for the insightful guidance and many valuable discussions on this project. We also thank Professors Jason Manning, Tongrui Wang, Feng Wang, Xiaolei Wu for helpful discussions on algebraic topology. Part of this work was done when X.L. visited Shanghai Center for Mathematical Sciences and he would like to express his gratitude for their hospitality. X.L. is supported by NSF grant DMS-1945178.

During the Workshop on Minimal Surfaces and Mean Curvature Flows (Beijing July 8-12), we were informed that Adrian Chu and Yangyang Li also proved a similar result in an independent work.

\section{Preliminaries}\label{sec: preliminaries} 
In this paper, we consider the three-sphere $S^3$. We also use $\mb S^3$ to denote the standard unit round three-sphere. Consider the spaces
\[
\mc S :=\{\phi(\mb S^2)\big|\phi: \mb S^2\to \mb S^3 \text{ is a smooth embedding}\},
\]
and 
\[
\mc T:=\{\phi(\mb T^2)\big|\phi:\mb T^2\to\mb S^3 \text{ is a smooth map embedding}\}.
\]
 Denote by $\oli{\mc S}$ the union of $\mc S$ and the space of embedded circles and points. We use $\oli{\mc T}$ to denote $\mc T\cup \oli{\mc S}$. Given $\ell >0$, denote by 
\begin{gather*}  
    \oli{\mc T}^\ell:=\{\Sigma\in \oli{\mc T}; \mc H^2(\Sigma)<\ell\};\quad    \oli{\mc S}^\ell:=\{\Sigma\in \oli{\mc S}; \mc H^2(\Sigma)<\ell\}.
\end{gather*}
Denote by $\mc T_{min}(g)$ the collection of embedded minimal tori w.r.t. the metric $g$. We always consider the smooth topology on $\mc T$ and $\omc S$ and extend this topology to $\omc T$.
\begin{definition}
We say that $\{T_i\}\subset \omc T$ converges to $T$ if and only if one of the following holds:
\begin{enumerate}
    \item $T$ is one dimensional and the convergence is smooth;
    \item $T_i$ and $T$ are all tori (or spheres) and the convergence is smooth in the classical sense;
    \item $T_i$ are tori, $T$ is a sphere, and $T_i$ converges to $T$ locally smoothly away from at most one point.
\end{enumerate}
Such a topology is called the {\em generalized smooth topology} on $\omc T$ (or just {\em smooth topology} for simplicity). 
\end{definition}

We recall some notations from Geometric Measure Theory. In the following, let $(M^3, g)$ denotes a closed, oriented, $3$-dimensional Riemannian manifold isometrically embedded in some $\mathbb{R}^J$, and $U \subset M$ an open subset ($U$ may be equal to $M$). 

\begin{itemize}
    \item $\mc H^2$ : the $2$-dimensional Hausdorff measure in the metric space $(M, d)$;
    \item $\mathbf{I}_{2}(M; \mathbb{Z}_2)$: the space of $2$-dimensional mod $2$ flat chains in $\mathbb{R}^J$ and support contained in $M$; 
    \item $\mathcal{Z}_{2}(M; \mathbb{Z}_2)$: the space of mod $2$ flat cycles or mod $2$ flat cycles; 
    \item $\mathcal{V}_2(M)$: the space of $2$-varifolds in $M$;
\end{itemize}

Given $T\in {\bf I}_2(M;\mathbb{Z}_2)$,  we denote by $|T|$ and $||T||$ the integral varifold and the Radon measure in $M$ associated with $|T|$, respectively;  For $V\in \mathcal{V}_2(M)$, $||V||$ denotes the Radon measure in $M$ associated with $V$. 

The {\it mass} of $T \in {\bf I}_2(M;\mathbb{Z}_2)$ is denoted by ${\bf M}(T)$, and the metric ${\bf M}(T_1,T_2)={\bf M}(T_1-T_2)$ defines the mass topology. The {\it flat metric} 
 $$
 \mathcal F(T_1,T_2) = \inf \{{\bf M}(P)+{\bf M}(Q): T_1-T_2=P+\partial Q\}
 $$
 induces the flat topology. The definition of ${\bf F}$-{\it metric} appeared in the book of Pitts \cite{Pi}*{page 66} and it induces the varifold weak topology on $\mathcal{V}_2(M)\cap \{V: ||V||(M) \leq A\}$ for any $A$. Note that  
$$||V||(M) \leq ||W||(M) +{\bf F}(V,W)$$ for all $V,W \in \mathcal{V}_2(M)$. Denote by ${\overline{\bf B}^{\bf F}_{\delta}(V)}$ and ${\bf B}^{\bf F}_{\delta}(V)$ the closed and open balls under $\mathbf{F}$-metric respectively, with radius $\delta$ and center $V \in \mathcal{V}_2(M)$. Similarly, denote by ${\overline{\bf B}^{\mathcal F}_{\delta}(T)}$ and ${\bf B}^{\mathcal F}_{\delta}(T)$ the corresponding balls under the flat metric with center $T \in \mathcal{Z}_2(M;\mathbb{Z}_2)$. 
Finally,  the ${\bf F}$-{\it metric} on ${\bf I}_2(M;\mathbb{Z}_2)$ is given by
$$ {\bf F}(S,T)=\mathcal{F}(S-T)+{\bf F}(|S|,|T|).$$
Since ${\bf F}(|S|,|T|) \leq {\bf M}(S,T)$, we know ${\bf F}(S,T) \leq 2{\bf M}(S,T)$ for any $S,T \in {\bf I}_2(M;\mathbb{Z}_2)$.

In this work, we equip ${\bf I}_2(M;\mathbb{Z}_2)$ and ${\mathcal Z}_2(M;\mathbb{Z}_2)$ with the topology induced by the flat metric. The space $\mathcal{V}_2(M)$ is considered with the weak topology of varifolds.

\subsection{Notations in min-max theory}
Now consider $M^3 = S^3$ equipped with a metric $g$. Let $X$ be a finite dimensional cubical complex, and $Z \subset X$ be a subcomplex.
Given $\Phi_{0}: X \rightarrow \oli{\mc T}$ a continuous map in the smooth topology, we denote by $\Pi$ the set of all continuous maps $\Phi: X \rightarrow \oli{\mc T}$ which is homotopic to $\Phi_0$ relative to $\Phi_{0}|_{Z}: Z \rightarrow \oli{\mc T}$. We refer to such a $\Phi$ a $k$-{\em sweepout}.  

\begin{definition}
    Given $(X, Z)$, $\Phi_0$, and $\Pi$ as above, the {\em min-max value} of $\Pi$ is defined by: 
    \[
    L = \mathbf{L}(\Pi) = \inf_{\Phi\in\Pi} \sup_{x \in X} \mathcal{H}^2(\Phi(x)).   
    \]
    A sequence $\{\Phi_{i}\}_{i \in \mathbb{N}} \subset \Pi$ is called a {\em minimizing sequence} if 
    \[
    \mathbf{L}(\Phi_{i}) := \sup_{x \in X} \mathcal{H}^2(\Phi_{i}(x)) \rightarrow L, \text{ when } i \rightarrow \infty. 
    \]
    A subsequence $\{\Phi_{i_j}(x_j): x_j \in X\}_{j \in \mathbb{N}}$ is called a {\em min-max (sub)sequence} if 
    \[
    \mathcal{H}^{2}(\Phi_{i_j}(x_j)) \rightarrow L, \text{ when } j \rightarrow \infty.
    \]
    The {\em critical set} of a minimizing sequence $\{\Phi_{i}\}$ is defined by 
    \[ \mathbf{C}(\{\Phi_i\})=\left\{V \in \mathcal{V}_2(M)\left|\,
    \begin{aligned}   
        & \exists \text{ a min-max subsequence }\{\Phi_{i_j}(x_j)\} \text{ such}\\
        & \text{that } \mathbf{F}(|\Phi_{i_j}(x_j)|, V) \to  0 \text{ as } j\to\infty
    \end{aligned}\right\}\right..
    \]
\end{definition}

By constructing a tightening map adapted to the area functional and applying it to a minimizing sequence, we can prove the following pull-tight theorem analogous to \cite{wangzc2023existenceFour}*{Theorem 2.7} and \cite{Colding-DeLellis03}*{Proposition 4.1}.

\begin{theorem}[Pull-tight] 
Let $(X, Z)$, $\Phi_0$, and $\Pi$ be as above. Given a minimizing sequence $\{\Phi_{i}^{*}\} \subset \Pi$ associated with the area functional, there exists another minimizing sequence $\{\Phi_{i}\} \subset \Pi$ such that $\mathbf{C}(\{\Phi_i\}) \subset \mathbf{C}(\{\Phi_i^{*}\})$ and every element $V \in \mathbf{C}(\{\Phi_i\})$ is either stationary, or belongs to $B = \Phi_{0}(Z) \subset \oli{\mc T}$.
\end{theorem}

\subsection{Interpolation theorem}
Recall that $d_H(\cdot,\cdot)$ is the Hausdorff distance between two sets.

\begin{proposition}\label{prop:interpolation to a fixed surface}
    Let $\Sigma$ be an embedded torus (or sphere) in $(S^3,g)$ and $\ell\in \mb N$. There exists $\eta_0=\eta_0(\Sigma, g)>0$ such that for every $0<\eta<\eta_0$, and every continuous map $\Phi: \partial I^\ell\to \omc T$ (resp. $\omc S$) with
\[    \mf F(\Phi(y),|\Sigma|)<\eta,
\]
there is an extension $\hat \Phi: I^\ell\to \omc T$ of $\Phi$ with 
\[
   \sup_{y\in I^\ell} \mf F(\hat \Phi(y),\Phi(y))  \leq \eta_0  .
\]
\end{proposition}

\begin{proof}
We refer to \cite{Liokomovich-Ketover23} for a detailed proof and one can check that the area varies slightly in a small neighborhood, which gives the upper bound for the $\mf F$ metric. Here we provide an alternate proof to show the area discrepancy when $\Sigma$ is a torus, which is based on the Simon-Smith min-max theory and \cite{Liokomovich-Ketover23}.

Suppose that $\Sigma$ is a Clifford torus in a round $S^3$. Then by Ketover-Liokumovich \cite{Liokomovich-Ketover23}*{Section 7}, there exists a homotopy $\psi:\partial I^\ell\times [0,1]\to \mc T$ of $\phi$  such that 
\begin{gather*} 
    \sup_{y\in\partial I^\ell} d_H(\psi (x,t) ,\Sigma)<\gamma \quad \text{ for all } (x,t)\in  \partial I^\ell,   \\
    \psi(\cdot,0)=\Phi(\cdot), \quad \psi(\cdot,1)=\Sigma.
\end{gather*}
Such a homotopy gives an extension $\Psi: I^\ell \to \mc T$ of $\Phi$ such that $\Psi|_{\partial I^\ell}=\Phi$. For a general tori, we consider a diffeomorphism of $S^3$ which maps $\Sigma$ to a Clifford torus. By taking the extension near the Clifford torus, one can obtain an extension $\Psi: I^\ell \to \mc T$ of $\Phi$ such that $\Psi|_{\partial I^\ell}=\Phi$.

Now we deform $\Psi$ to be a map which also has desired area upper bounds. By Ketover-Liokumovich \cite{Liokomovich-Ketover23}*{Section 7}, we can suppose that $d_H(\Psi(x),\Sigma)<\gamma$.

Since $X$ is compact, there exists $L>0$ such that 
\[  
    \mc H^2_g(\Psi(x)) <  L. 
\]
Now we consider the conformal metric $\wti g=e^{2u}g$ for some $u\in C^\infty(S^3)$. Then the second fundamental form of $\Sigma$ w.r.t. $\wti g$ is given by (Besse \cite{Bes87}*{Section 1.163})
\[   A_{\Sigma,\wti g}=A_{\Sigma,g}+g\cdot (\nabla u)^\perp,\]
where $(\nabla u)^\perp$ is the component of $\nabla u$ normal to $T_x\Sigma$. We can pick $u$ such that $\Sigma$ is a strictly stable minimal surface and $|u|<\eps$ for any given $\eps$. It follows that 
\[  e^{-2\eps} \mc H^2_{g}(S)< \mc H^2_{\wti g}(S)\leq e^{2\eps} \mc H^2_{g}(S).\]
Let $\gamma>0$ such that $\Sigma\times\{r\}$ has mean curvature vector pointing towards $\Sigma$ for $-2\gamma\leq r\leq 2\gamma$. By taking small $\eta$, we can also suppose that $\Psi$ has image lying in $\Sigma\times (-\gamma,\gamma)$.

Note that Simon-Smith min-max theory applies to $\Psi$ relative to $\Psi|_{\partial I^\ell}$ in the space $(\Sigma\times [2\gamma,2\gamma],\wti g)$ since the boundary has mean curvature vector pointing towards $\Sigma$. Then there are two possibilities: either $\Psi$ is homotopic a map $\Psi':X\to \mc T$ with 
\[    \mc H^2(\Psi'(x);\wti g) < \max\Big\{ \sup_{y\in \partial I^\ell} \mc H^2(\Psi(y);\wti g), \mc H^2(\Sigma;g) \Big\} \]
or such a family is associated with an embedded minimal surface $\Gamma$ with
\[  
    \quad   \mc H^2(\Gamma;g)< L;\quad  \mk g(\Gamma)\leq 1.   \]
Suppose that the first case happens. Then we can push $\Psi'$ has the desired area upper bound.
    
Now we consider the second case. Since $\Sigma$ is the unique minimal surface, then $\Gamma$ must be $m\Sigma$ for a positive integer multiplicity $m$. On the other hand, since $\Sigma$ is a torus, then the genus upper bound \cite{Ketover13} gives that $m=1$. This implies we can deform $\Psi$ to $\Psi'$ which has desired area upper bound. Then we get the desired extension.
\end{proof}

Then by a subdivision to the domain, one can inductively prove the following interpolation result; see \cite{MN18}*{Theorem 3.8}.
\begin{lemma}\label{lem: interplation result}
Let $\mc K$ be a compact set of smooth tori or spheres and $k\in \mb N$. For every $\alpha > 0,r>0$, there exists $0 < \beta < \alpha$ (depending on $\mc K$) such that for any cubical complex $X$ of dimension at most $k$ and any pair of continuous maps $\Phi, \Psi: X \rightarrow \oli{\mc T}$ in the smooth topology with $\Psi(X) \subset \mathbf{B}_{\beta}^{\mathbf{F}}(\mc K)$ and 
\[
\sup_{x \in X} \mathbf{F}(\Phi(x), \Psi(x)) < \beta, 
\]
we can find a homotopy $H:[0,1]\times X \rightarrow \oli{\mc T}$ with $H(0, \cdot) = \Phi$, $H(1, \cdot) = \Psi$, and such that
\[
    \mathbf{F}(H(t, x), \Psi(x)) < \alpha \quad \text{ for all } (t,x)\in [0,1]\times X.
\]
\end{lemma}
\begin{proof}
The proof is similar to \cite{MN18}*{Proposition 3.6} and we sketch the idea here for completeness.

We proceed by an induction on $k$. When $k=0$, $X$ is a finite set. For every $T\in \mc K$, let $\eta(T)$ be the constant in Proposition \ref{prop:interpolation to a fixed surface} (where we used $\eta_0(T)$ therein) so that any map into $\mf B^\mf F_{\eta(T)}(T)$ is homotopic to the constant map to $T$. Since $\mc K$ is compact, there exists a finite cover $\{\mf B^\mf F_{\eta(T_i)}(T_i)\}$ of $\mc K$. Now set $\eta:=\min\{\eta(T_i)\}$. Then given any $x\in X$, there exists $T_i$ such that $\Phi(x),\Psi(x)\in \mf B^\mf F(T_i)$. By Proposition \ref{prop:interpolation to a fixed surface}, there exists a homotopy connecting $\Phi$ and $\Psi$. Clearly, this is the desired homotopy.

For $k\geq 1$, the inductive hypothesis gives a homotopy between $\Phi$ and $\Psi$ on the $(k-1)$ skeleton of $X$. Applying barycentric subdivisions to $X$, we can make sure for any $k$-simplex $\mk t$, $\Phi(x), \Psi(y)\in \mf B^\mf F_{\eta(T_i)}(T_i)$ for some $T_i$ whenever $x,y\in\mk t$. Then by Proposition \ref{prop:interpolation to a fixed surface}, one can extend the homotopy on each $[0,1]\times \mk t$. This finishes the proof of Lemma \ref{lem: interplation result}.
\end{proof}

\subsection{Deformation theorem}\label{sec: deformation theorem}
By a deformation theorem analogous to Deformation Theorem A in \cite{Marques-Neves16}, one can prove the Morse index upper bound in the Simon-Smith setting \cites{danielxin2015, Marques-Neves16}. To show the Morse index lower bound, we first need to demonstrate the existence of minimizing sequences such that every element of the critical set is almost smooth. This relies on a smooth version of deformation theorem analogous to Pitts' Combinatorial Deformation Theorem in \cite{Pi}. In the following, let $\mathcal{W}_{L}$ be the collection of all stationary integral varifolds in $S^3$ with mass equal to $L$ whose support is a closed smooth embedded minimal torus or minimal sphere. Denote $\mathcal{W}_{L, j}, \mathcal{W}^{j}_{L}$ the elements in $\mathcal{W}_{L}$ whose support has Morse index less than or equal to $j$ and bigger than or equal to $j$ respectively. 

In \cite{wangzc2023existenceFour}*{Theorem 7.3, Step 1}, the second author together with X. Zhou proved that for a homotopy class $\{\Phi_i\}\subset \Pi$, by doing an extra tightening process to find another minimizing sequence, still denoted as $\{\Phi_i\}$, such that for $i$ sufficiently large, either $\Phi_i(x)$ is close to a smooth min-max minimal surface, or the area $\mc H^2 (\Phi_i(x))$ is strictly less than ${\mf L}(\Pi)$. 
\begin{theorem}[\cite{wangzc2023existenceFour}*{Theorem 7.3, Step 1}]\label{thm: Deform away from nonsmooth critical varifolds}
Suppose $g$ is a bumpy metric on $S^3$. Consider $\Pi$ a homotopy class of $k$-sweepouts by elements in $\oli{\mc T}$ and a minimizing sequence $\{\Psi_i\} \subset \Pi$ such that 
\[
\mathbf{C}(\{\Psi_i\}) \cap \mathcal{W}^{k + 1}_{L} = \emptyset. 
\]
Let $\Lambda = \mathbf{C}(\{\Psi_i\}) \cap \mathcal{W}_{L} \subset \mathcal{W}_{L, k}$. Given any $h > 0$, there exists another minimizing sequence $\{\hat{\Psi}_{j}\} \subset \Pi$ such that $\mathbf{C}(\{\hat{\Psi}_i\}) \cap \mathcal{W}^{k + 1}_{L} = \emptyset$ and for some $\xi \in (0, L)$ and $i$ large enough, 
\[
\{|\hat{\Psi}_{i}(x)|: \mc H^2 (\hat{\Psi}_{i}(x)) \geq L - \xi\} \subset \bigcup_{\Sigma \in \Lambda} \mathbf{B}^{\mathbf{F}}_{h}(\Sigma).  
\]
\end{theorem}

The next theorem concerns sequences of mod $2$ flat cycles whose flat limit is different from the varifold limit.  

\begin{theorem}[\cite{MN18}*{Theorem 4.11}]\label{thm: deform the sweepout if all its values are close to a minimal cycle and a different current} 
Let 
\[
\Sigma = |\Sigma_1| + \cdots + |\Sigma_N| 
\]
be an embedded minimal varifold where $\{\Sigma_1, \ldots, \Sigma_N\}$ is a disjoint collection, and   
\[
T = n_1 \cdot \Sigma_1 + \cdots n_{N} \cdot \Sigma_{N}, \quad n_i \in \{0, 1\}
\]
be a mod $2$ flat cycle with $n_j \neq 1$ mod $2$ for some $1 \leq i \leq N$. 

Given $\Psi: X \rightarrow \oli{\mc T}$ a continuous map in the smooth topology, where $X \subset I(l, j)$ is a subcomplex, there exists $\varepsilon = \varepsilon(l, \Sigma, T) > 0$ such that if the following holds for every $x \in X$
\[
\mathbf{F}(|\Psi(x)|, \Sigma) \leq \varepsilon \text{ and } \mathcal{F}(\Psi(x), T) \leq \varepsilon, 
\]
then for every $\delta > 0$ we can find a homotopy $H: [0, 1] \times X \rightarrow \oli{\mc T}$ satisfying 
\begin{enumerate}
    \item $H(0, \cdot) = \Psi$;
    \item $H(1, \cdot) = \Psi'$;
    \item $\mathcal{H}^2(H(t, x)) \leq \mathcal{H}^2(\Psi(x)) + \delta$ for every $(t, x) \in [0, 1] \times X$;
    \item $\mathcal{H}^2(\Psi'(x)) \leq \mathcal{H}^2(\Psi(x)) - \varepsilon/6$ for every $x \in X$.
\end{enumerate} 
\end{theorem}

\begin{proof}
    As the proof shares the same spirit as that of \cite{MN18}*{Theorem 4.11}, we only sketch it here. Let $r > 0$ be a small number. Applying barycentric subdivisions to $X$, we assume that 
    \[
    |\mathcal{H}^2(\Psi(x)) - \mathcal{H}^2(\Psi(y))| \leq \delta  
    \]
    whenever $x, y$ live in a common cell of $X$. Choose $p \in \Sigma_i$ and $\{p_1, \ldots, p_k\} \subset \Sigma_i$ such that $|p_j - p|$ is sufficiently small. We claim that there exists $\varepsilon = \varepsilon(l, \Sigma, T) > 0$ so that for every $x \in X$, $\Psi(x)$ admits an $(\varepsilon, \delta)$-deformation in $B_{r}(p_j)$ for all $j = 1, \ldots, k$, i.e. there exists an isotopy $\phi \in \Is(B_{r}(p_j))$ such that 
    \begin{itemize}
        \item $\mathcal{H}^2(\phi(t,\Psi(x))) \leq \mathcal{H}^2(\Psi(x)) + \delta \quad \forall t \in [0, 1]$;
        \item $\mathcal{H}^2(\phi(1, \Psi(x))) \leq \mathcal{H}^2(\Psi(x)) - \varepsilon$.
    \end{itemize}
    Indeed, since the varifold limit of an almost minimizing sequence of embedded surfaces has the smooth regularity by the Simon-Smith min-max theory, the proof of \cite{MN18}*{Proposition 4.10} would carry over with our $\Psi(x) \in \oli{\mc T}$ by using \cite{Colding-DeLellis03}*{Theorem 7.1} in place of \cite{Pi}*{Section 3.10}. 

    Now, following the proof of \cite{wangzc2023existenceFour}*{Lemma 3.10} (see also \cite{Colding-Gabai_Ketover18}*{Appendix}), we obtain the desired homotopy $H$. 
\end{proof}

\section{Relative homology min-max theory}\label{sec: relative homology min-max theory}

\subsection{Index lower bound for Simon-Smith homology min-max theory}
By Marques-Neves \cite{Marques-Neves16}*{\S 1.3}, the minimal surfaces produced by Simon-Smith min-max theory have Morse index upper bound; see also \cite{danielxin2015}. In this subsection, we establish the Morse index lower bound (Theorem \ref{thm: index lower bound}) in the Simon-Smith setting, which relies on a combination of White's Local Min-max Theorem \cite{MN18}*{Theorem 6.1} and Theorem \ref{thm: Deform away from nonsmooth critical varifolds}.   

We recall the following property for generic metrics.    
\begin{proposition}[\cite{MN18}*{Proposition 8.5}]\label{prop: a choice of generic metric}
    For a $C^\infty$-generic positive Ricci metric $g$ on $M^3$, we have: 
    \begin{enumerate}
        \item every closed $g$-minimal surface is non-degenerate and unstable; 
        \item and if
    \[
    p_1 \cdot \mc H^2_g(\Sigma_1) + \cdots + p_{r} \cdot \mc H^2_g(\Sigma_r) = 0,
    \] 
    with $\{p_1, \ldots, p_r\} \subset \mathbb{Z}$, $\{\Sigma_1, \ldots, \Sigma_r\}$ a collection of connected $g$-minimal surfaces, and $\Sigma_k \neq \Sigma_l$ whenever $k \neq l$, then
    \[
    p_1 = \cdots = p_r = 0. 
    \]
    \end{enumerate}
\end{proposition}

Recall the definitions of $\mathcal{W}_{L}$, $\mathcal{W}_{L, j}$, and $\mathcal{W}^j_{L}$ in Section \ref{sec: deformation theorem}. In the next result, we prove the index lower bound for Simon-Smith homology min-max theory.   

\begin{theorem}\label{thm: index lower bound} 
Suppose $g$ is a generic metric on $S^3$ with positive Ricci curvature as in Proposition \ref{prop: a choice of generic metric}. Let $\tau\in H_k(\oli{\mc T};\mb Z_2)$ be a nontrivial homology class, and let 
\[  W(\tau) = \inf_{[\sum \mk s_i] = \tau} \sup_{i, x \in \Delta^k} \mathcal{H}^2(\mk s_i(x)). \]
Then there exists a two-sided, embedded, minimal torus or minimal sphere $\Sigma$ with 
\[  \mc H^2(\Sigma)=W(\tau), \quad  \mr {index}=k.\]
\end{theorem}

\begin{proof}
We closely follow the proof of \cite{MN18}*{Theorem 7.2}, with some alterations. First, we claim that $W(\tau) > 0$. Take $\{\Phi_i:X_i\to\omc T\}$ such that $(\Phi_i)_*[X_i]=\tau\in H_k(\omc T;\mb Z_2)$ and 
\[W(\tau)=\lim_{i\to\infty} \sup \{\mc H^2(\Phi_i(x));x\in X_i\}.\]
Then by Simon-Smith min-max theory (c.f. \cites{wangzc2023existenceFour, Ketover13}), there exists an embedded minimal sphere or torus $S_i$ with 
\[  \mc H^2(S_i) \leq \sup \{\mc H^2(\Phi_i(x));x\in X_i\}.\]
By compactness of minimal surfaces with bounded topology, $\mc H^2(S_i)$ has a positive lower bound. It follows that $W(\tau)>0$.

We proceed to show the index lower bound. Let $\{\sum_{i} \mk s_{i}^{j}\}_{j}$ be a minimizing sequence of representatives of $\tau$, i.e.
    \[
    [\sum_{i} \mk s_{i}^{j}] = \tau \quad \forall j \quad \text{and} \quad
    \lim_{j \rightarrow \infty} \sup_{i, x \in \Delta^k} \mathcal{H}^2(\mk s_i^j(x)) = W(\tau). 
    \]
    By Section 2.1 of \cite{Hat-AT-book}, for each $\sum_{i} \mk s_{i}^j$, there exist a corresponding $k$-dimensional $\Delta$-complex $X$ ($X$ depends on $j$) and a map $\Phi_j : X \rightarrow \cup_{i} \mk s_{i}^j(\Delta^k) \subset \oli{\mc T}$ that is continuous in the smooth topology. Moreover, we have
    \[
    \quad (\Phi_j)_{*}[X] = \tau. 
    \]

Since the metric $g$ is as in Proposition \ref{prop: a choice of generic metric}, there exists $\alpha > 0$ such that every embedded minimal cycle in $\overline{\mathbf{B}}^{\mathbf{F}}_{\alpha}(\mathbf{C}(\{\Phi_j\})) \cap \mathcal{W}_{L, k}$ has multiplicity one. By deforming $\{\Phi_j\}$ to avoid minimal surfaces of higher index, we construct a pulled-tight minimizing sequence $\{\Psi_j\}$ such that $\mathbf{C}(\{\Psi_j\}) \cap \mathcal{W}^{k + 1}_{L} = \emptyset$ and
     \[
     \mathbf{C}(\{\Psi_j\}) \subset \overline{\mathbf{B}}^{\mathbf{F}}_{\alpha}(\mathbf{C}(\{\Phi_j\})). 
     \]

Without loss of generality, we assume that $\Lambda := \mathbf{C}(\{\Psi_j\}) \cap \mathcal{W}_{L, k} = \{\Sigma\}$. By contradiction, $\Index(\Sigma)=m < k$. Let $\varepsilon_{0}(\Sigma) > 0$ be the constant in White's Local Min-max Theorem (see \cite{MN18}*{Theorem 6.1}) associated with $\Sigma$, i.e. for every $S\in \omc T$ with $\mc F (S,\Sigma)<\eps_0$, we have 
     \[   
\max_{v\in \oli B^m} \|(F_v)_\#S\|(S^3)\geq \mc H^2(\Sigma),
     \]
     where $\{F_v\}_{v\in \oli B^m}\subset \mr{Diff}(S^3)$ generated by the first $j$ eigenfunctions of Jacobi operator of $\Sigma$; see \cite{MN18}*{Theorem 6.1} for more details. Moreover, there exists $\mu>0$ such that 
     \begin{equation}\label{eq:def of mu}
         \mc H^2( F_v(\Sigma))<\mc H^2(\Sigma)-2\mu.
     \end{equation} 
     
     Let $\delta' > 0$ be the constant of Lemma \ref{lem: interplation result}.
     Given any 
    \[
    0 < h_1 < \min\{\varepsilon_0(\Sigma), \delta'\}   
    \]
    we may ensure that 
    for each $T \in\{0,\llbracket\Sigma\rrbracket\}\subset \mc Z_2(S^3;\mb Z_2)$,   
    \[
    \mathbf{B}_{2h_1}^{\mathbf{F}}(\Sigma) \cap \Lambda = \{\Sigma\}, \quad \mathbf{B}_{2h_1}^{\mathcal{F}}(T) \cap \mathcal{Z} = \{T\}.   
    \]
     
   Given $0<h<\min\{h_1,\mu/10\}$,  Theorem \ref{thm: Deform away from nonsmooth critical varifolds}  gives a pulled-tight minimizing sequence $\{\hat{\Psi}_j\} \subset \Pi$ such that the only smooth element of $\mathbf{C}(\{\hat{\Psi}_j\})$ is $\Sigma$ and for some $\xi \in (0, h)$, 
\[
x\in X,\mathcal{H}^2(\hat{\Psi}_{j}(x)) \geq L - \xi\quad  \Longrightarrow \quad \mf F(|\hat \Phi_j(x)|,|\Sigma|)<h.  
\]  
Because of Theorem \ref{thm: deform the sweepout if all its values are close to a minimal cycle and a different current}, we can also suppose 
\[ x\in X, \mathcal{H}^2(\hat{\Psi}_{j}(x)) \geq L - \xi\quad \Longrightarrow \quad \mc F(\hat \Phi_j(x),\Sigma)<h. \]
    Now fix a large $j$. We denote by $X^k$ the $k$-dimensional skeleton of $X$ and $\widetilde{X}^k$ the union of all $k$-simplicies of $X^k$. By Mayer-Vietoris sequence, one can assume that $X^k=\wti X^k$.

    Choose $0 < \delta < \xi/4$. Applying subdivisions to $X^k$ (still denoted by $X^k$) with 
    \[\mathbf{F}(\hat{\Psi}_j(x), \hat{\Psi}_j(y)) < \delta\]
    whenever $x, y$ live in a common cell of $X^k$. Let $V$ be the union of all {$k$-simplicies } $\mk s \in X^k$, 
    \[ \mathcal{H}^2(\hat{\Psi}_j(y)) \geq L - \xi/2\]  
    for some $y \in \mk s$.
   Then we have that 
    \begin{equation*}
        \mathbf{F}(|\hat{\Psi}_j(y)|, \Sigma) < h, \quad \mathcal{F}(\hat{\Psi}_j(y), \Sigma) < h \quad \forall y \in W.
    \end{equation*}
    Moreover, if $y\in \partial V$, then
    \begin{equation}\label{eq: deformed slice} 
        L - \xi \leq \mathcal{H}^2(\hat{\Psi}_j(y)) < L -\xi/2 \quad \forall y \in \tau.
    \end{equation}
    
    
    Recall $\{F_v\}_{v \in \overline{B}^m} \subset \Diff(S^3)$ associated with $\Sigma$ in White's Local Min-max Theorem \cite{MN18}*{Therem 6.1}. For $y \in \partial V$, define the function $A^y: \overline{B}^m \rightarrow [0, \infty)$ by 
    \[A^y(v) = \mathcal{H}^2(F_v(\hat{\Psi}_j(y))), 
    \]  
    and let $\{\psi^{y}(\cdot, t)\}_{t \geq 0} \subset \Diff(\overline{B}^m)$ be the one-parameter negative gradient flow generated by the vector field 
    \[v \mapsto -(1 - |v|^2)\nabla A^y(v) \text{ with } \lim_{t \rightarrow \infty} \psi^y(v, t) \in \partial B^m.\]

    Recall that in the proof of \cite{MN18}*{Theorem 7.2}, they defined three homotopies $H_1$, $H_2$, and $H_3$, which we will adopt with certain adjustments. Specifically, there exists homotopy  
    \[H_1: [0, 1] \times \partial V \rightarrow \oli{\mc T},\] 
    which is defined by 
    \[
    H_1(t, y) = F_{\psi^y(0, t_0 t)}(\hat{\Psi}_j(y))
    \]
    for some $t_0 > 0$ with the properties
    \begin{itemize}
        \item $\mathbf{F}(H_1(1, y), F_{w(y)}(\Sigma )) \leq 5h_1 \quad \forall y \in \partial V$ for some continuous function $w: \partial W_{j, p} \rightarrow \partial B^m$; 
        \item $\mathcal{H}^2(H_1(t, y)) < L - \mu + 5h_1<L-3\mu/2 \quad \forall (t, y) \in [0, 1] \times \partial V$.
    \end{itemize}  
    Here $\mu > \xi$ is defined in \eqref{eq:def of mu}.

    By Lemma \ref{lem: interplation result} with $\mathcal{K} = \{F_w(\Sigma): w \in \overline{B}^m\} \subset \oli{\mc T}$, we get that there exists a second homotopy
    \[H_2: [1, 2] \times\partial V \rightarrow \oli{\mc T},\]
    which interpolates between $H_1(1, \cdot)$ and $F_{w(\cdot)}(\Sigma)$ with the properties
    \begin{itemize}
        \item $\mathbf{F}(H_2(t, y), F_{w(y)}(\Sigma)) \leq \mu/8 \quad \forall (t, y) \in [1, 2] \times \partial V$;
        \item $\mathcal{H}^2(H_2(t, y)) < L - \mu \quad \forall (t, y) \in [1, 2] \times \partial V$.  
    \end{itemize}
    
    Note that $\Index(\Sigma) = m \geq 1$ and $m<k$, then $H_{k - 1}(\partial B^m; \mathbb{Z}_2) = 0$ and hence 
    \[
   [ w_{\#}(\partial V)] = 0. 
    \]
    This further suggests that we can fill in $w_{*}[\partial V]$ by a $k$-chain $Q$ in $\partial B^m$.
Consider the following $\Delta$-complex:
\[ C:=(\hat \Psi_j)_\# V + (H_1)_\#([0,1]\times \partial V)+(H_2)_\#([1,2]\times \partial V) + F_Q(\Sigma).\]
Clearly, $\partial C=0$ and $C\subset \mf B^\mf F_\mu (\Sigma)$. Hence $[C]=0\in H_k(\omc T;\mb Z_2)$ by Lemma \ref{lem: interplation result} with $\mathcal{K} = \{\Sigma\}$. Hence 
\[    
\wti Z : = (\hat \Psi_j)_\# V + (H_1)_\#([0,1]\times \partial V)+(H_2)_\#([1,2]\times \partial V) + F_Q(\Sigma).
\]
is a $k$-cycle and $[\wti Z]=\tau \in H_k(\omc T;\mb Z_2)$. By the construction, we have that 
\[
\mc H^2(\Gamma)<L-\frac{\xi}{2} \quad \text{ for all } \Gamma \in \image{\wti Z}.
\]
which leads to a contradiction. 
\end{proof}

\subsection{Relative homology min-max theorem}
Recall that given $\ell >0$,
\begin{gather*}  
    \oli{\mc T}^\ell:=\{\Sigma\in \oli{\mc T}; \mc H^2(\Sigma)<\ell\};\quad    \oli{\mc S}^\ell:=\{\Sigma\in \oli{\mc S}; \mc H^2(\Sigma)<\ell\}.
\end{gather*}
Recall that $\mc T_{min}(g)$ the collection of embedded minimal tori w.r.t. the metric $g$. 
\begin{theorem}[Relative Homology Min-max Theorem]\label{thm:relative homology min-max}
Let $g$ be a bumpy metric. Let $\tau\in H_k(\omc T^t, \omc T^s\cup \omc S^t)$ be a non-trivial homology class, $0\leq s<t$, and let 
\[ 
    W(\tau) = \inf_{[\sum \mk s_i] = \tau} \sup_{i, x \in \Delta^k} \mathcal{H}^2(\mk s_i(x)). \]
Suppose that, for some $\eps>0$, there is no embedded minimal torus $\Sigma'\in \mc T_{min}(g)$ with $\mc H^2(\Sigma')\in (s-\eps,s)$ and $\mr{index}(\Sigma')\leq k-1$ and there is no embedded minimal sphere $S'$ with $\mc H^2(S')=s$. Then $W(\tau)\in [s,t)$ with $W(\tau)>0$ if $s=0$. Moreover there exists $\Sigma \in \mc T_{min}(g)$ with $\mr{index}=k$ and 
\[    \mc H^2(\Sigma)=W(\tau).\]    
\end{theorem}

\begin{proof} 
Clearly, we have $W(\tau) \in [s, t)$. The fact that $W(\tau) > 0$ if $s = 0$ follows similarly as in the proof of Theorem \ref{thm: index lower bound}.  We assume that $g$ is as in Proposition \ref{prop: a choice of generic metric}. 

Let $\{\Psi_j:(X,\partial X)\to (\omc T^t,\omc T^s\cup \omc S^t)\}$ ($X$ depends on $j$) be a minimizing sequence of representatives of $\tau$, i.e.
    \[
    (\Psi_j)_{*}[X] = \tau \quad \forall j \quad \text{and} \quad
    \lim_{j \rightarrow \infty} \sup_{x \in X} \mathcal{H}^2(\Psi_j(x)) = W(\tau). 
    \]

    For each $j$, consider $\Pi^j_{\partial}$ the homotopy class of ${\Psi_j}|_{\partial X}$ among all maps taking values in $\omc T^s\cup \omc S^t$. The min-max value to $\Pi^{j}_{\partial}$ is 
    \[
    \mathbf{L}(\Pi^j_{\partial}) = \inf_{\Phi \in \Pi^j_{\partial}} \,\sup_{y \in \partial X} \mathcal{H}^2(\Phi(y)) \in [0, t).   
    \]
    Depending on the value of $\lim_{j\to\infty}\mathbf{L}(\Pi^j_{\partial})$, we divide the proof into three cases:

\noindent    \textbf{Case I}: {\em Suppose that $\lim_{j\to\infty}\mathbf{L}(\Pi^j_{\partial}) < W(\tau)$.}

\smallskip 
Then we may assume that 
    \[
    \sup_{y \in \partial X^j} \mathcal{H}^2(\Psi_j(y)) < W(\tau)-\eps. 
    \]
 For simplicity, we omit the index $i$ in the following proof.

    Let $\Pi_{j}$ be the homotopy class of $\Psi_j$ among all maps relative to ${\Psi_j}|_{\partial X}$. The min-max value to $\Pi_j$ is 
    \[
    \mc L_j = \inf_{\Phi \in \Pi_j} \sup_{x \in X} \mathcal{H}^2(\Phi(x)) \in [s, t). 
    \] 
Then by Simon-Smith min-max theory, $\mc L_j$ is the area of minimal spheres or tori and $\mc L_j\to W(\tau)$. Thus, for sufficiently large $j$, $\mc L_j=W(\tau)$. Then we assume that $\{\Psi_j\}$ are in the same homotopy class $\Pi$. By \cite{wangzc2023existenceFour}, one can also assume that there exists a minimizing sequence (still denoted by $\{\Psi_j\}$) so that $\mathbf{C}(\{\Psi_j\})$ contains a multiplicity one embedded minimal torus or minimal sphere $\Sigma$. Since $g$ is as in Proposition \ref{prop: a choice of generic metric}, we assume that $\Sigma$ is the unique smooth element of $\mathbf{C}(\{\Psi_j\})$. From Theorem \ref{thm: index lower bound} and the fact that $g$ is bumpy, we have $\Index(\Sigma) = k$. 
    
Suppose by contradiction that $\Sigma$ is an embedded minimal sphere. Now let $\mu$, $h_1$ be the constants in Theorem \ref{thm: index lower bound}. 

Given $0<h<\min\{h_1,\mu/10,\eps/2\}$, as in Theorem \ref{thm: index lower bound}, there is a pulled-tight minimizing sequence $\{\hat{\Psi}_j\} \subset \Pi$ such that the only smooth element of $\mathbf{C}(\{\hat{\Psi}_j\})$ is $\Sigma$ and for some $\xi \in (0, h)$, 
\[
x\in X,\mathcal{H}^2(\hat{\Psi}_{j}(x)) \geq W(\tau) - \xi\quad  \Longrightarrow \quad \mf F(|\hat \Psi_j(x)|,|\Sigma|)<h.  
\]  
We can also suppose 
\[ x\in X, \mathcal{H}^2(\hat{\Psi}_{j}(x)) \geq W(\tau) - \xi\quad \Longrightarrow \quad \mc F(\hat \Psi_j(x),\Sigma)<h. \]
Since $\xi<h<\eps/2$, we also have that 
\[   
x\in X, \mathcal{H}^2(\hat{\Psi}_{j}(x)) \geq W(\tau) - \xi \geq W(\tau) - \frac{1}{2}\eps \quad \Longrightarrow x\notin \partial X.
\]
Applying subdivisions to $X$ (still denoted by $X$) with 
    \[\mathbf{F}(\hat{\Psi}_j(x), \hat{\Psi}_j(y)) < \xi/4\]
whenever $x, y$ live in a common cell of $X$. 
Let $V$ be the union of all $k$-simplicies in $X$, 
\[ \mathcal{H}^2(\hat{\Psi}_j(y)) \geq W(\tau) - \xi/2\]  
for some $y \in \mk s$. Then for any $x\in V$, 
    \[  \mc H^2(\hat\Psi_j(x)))  \geq   W(\tau)-\xi , \]
which implies that $V\cap \partial X=\emptyset $.

By the same argument in Theorem \ref{thm: index lower bound},
there exists $H:[0,2]\times \partial V\to \omc T^t$ (by combining $H_1$ and $H_2$ therein) such that for all $(t,y)\in [0,2]\times \partial V$,
\begin{gather*}
     \mc H^2(H(t,y))<W(\tau) -\xi/2 ; \quad  \mf F(H(t,y),\Sigma)<\mu/8,
\end{gather*}  
and $H(2,y)=F_{w(y)}(\Sigma)$ is an embedded sphere, where $(F_v)_{v\in \oli B^k}$ is from White's local min-max; see \cite{MN18}*{Theorem 6.1}.

Now consider 
\[  Y=  V\cup ([0,2]\times \partial V).  \]
Define the continuous map $\phi: (Y, \{2\}\times \partial V)\to (\omc T^t,\omc S^t )$ by 
\[  
\phi (x)=\hat{\Psi}_j(x) \quad \text { for } x\in V; \quad \phi (t,y)=H(t,y) \quad \text{ for } (t,y)\in [0,2]\times \partial V.
\]
Since $\mf F(\phi (y),\Sigma)\leq \mu/4$, the interpolation theorem \ref{lem: interplation result} with $\mathcal{K} = \{\Sigma\}$ yields that
\[\phi_{\#}[Y]=0\in H_k(\omc T^t,\omc T^s\cup \omc S^t;\mb Z_2).\]

Now let $Z= {(X\setminus V)} \cup ([0,2]\times \partial V)$ and define the continuous map 
\[\psi: (Z, \partial X\cup (\{2\}\times \partial V))\to (\omc T^t,\omc T^s\cup \omc S^t)\]
by 
\[  
\psi (x)=\hat{\Psi}_j(x) \quad \text { for } x\in X\setminus V; \quad \psi (t,y)=H(t,y) \quad \text{ for } (t,y)\in [0,2]\times \partial V.
\]
Clearly, $\phi_{\#}[Y]+\psi_{\#}[Z]=(\hat \Psi_j)_{\#}[X]$. This together with the fact of $\phi_{\#}[Y]=0$ gives that  
\[ 
\psi_{\#}[Z]= (\hat \Psi_j)_{\#}[X] = \tau. 
\]
On the other hand, the construction implies that 
\[   
  \mc H^2(\psi(x))\leq W(\tau)-\xi/2, 
\] 
which leads to a contradiction.

\medskip
\noindent\textbf{Case II}: {\em Suppose that $\lim_{j\to\infty}\mathbf{L}(\Pi^j_{\partial}) = W(\tau)> s$.}

\smallskip
Then by  Simon-Smith min-max theory together with the genus bound, index upper bound, multiplicity one theorem, there exists an embedded minimal sphere $\Sigma_j$ with 
\[  \mc H^2(\Sigma_j)=\mf L(\Pi^j_{\partial}), \quad \mr{index}(\Sigma_j)\leq k-1.\]
Since $g$ is a metric as in Proposition \ref{prop: a choice of generic metric}, then there are only finitely many minimal spheres with area upper bound. Hence we can also assume $\Sigma_j=\Sigma$ for all $j$. Let $\mu$ be the constant in \eqref{eq:def of mu}. We assume that $2\mu<W(\tau)-s$.

By the same argument as in Theorem \ref{thm: index lower bound} and {\bf Case I} above, we can also suppose that, given $0<h<\mu/10$, then there exists $\xi \in (0, h/4)$ such that 
\[
x\in\partial X,\mathcal{H}^2({\Psi}_{j}(x)) \geq W(\tau) - 2\xi\quad  \Longrightarrow \quad \mf F(|\Psi_j(x)|,|\Sigma|)<h/4 \text{ and } \mc F(\Psi_j(x),\Sigma)<h/4.  
\]  
Consider a subdivision of $X$ (still denoted by $X$) such that for each $k$-simplex $\mk t\in X$, we have 
\[   \mf F(\Psi_j(x),\Psi_j(y))<\xi/2 \quad \text{ for all }x,y\in\mk t.\]
Then for each simplex $\mk t$ that intersects $\partial X$, we have two possibilities: either 
\begin{gather*}
    \mc H^2(\Psi_j(x))\geq W(\tau)-2\xi\quad \text{ for all } x\in \mk t, \text{ or }\\
    \mc H^2(\Psi_j(x))\leq  W(\tau)-\frac{3}{2}\xi\quad \text{ for all } x\in \mk t.
\end{gather*} 
Note that in the first case, we get that 
\[   
    \mf F(|\Psi_j(x)|,|\Sigma|)<h/4 \text{ and } \mc F(\Psi_j(x),\Sigma)<h/4 \quad \text{ for all } x\in \mk t.
\]
Then by doing a second pull tight process (away from the simplex that intersects $\partial X$) as in \cite{wangzc2023existenceFour}*{Theorem 7.3, Step I} (see Theorem \ref{thm: Deform away from nonsmooth critical varifolds} and the Deformation Theorem \ref{thm: deform the sweepout if all its values are close to a minimal cycle and a different current}, we can also suppose that 
\[
x\in\partial X,\mathcal{H}^2({\Psi}_{j}(x)) \geq W(\tau) - \xi\quad  \Longrightarrow \quad \mf F(|\Psi_j(x)|,|\Sigma|)<h,\, \mc F(\Psi_j(x),\Sigma)<h.  
\]
Let $V$ be the union of all $k$-simplices $\mk t\in X$ such that there exists $x\in \mk t$ with 
\[   \mc H^2(\Psi_j(x))\geq W(\tau)-\frac{1}{2}\xi .   \]
Denote by $\partial_TV:=\oli{\partial V\setminus \partial X}$. Then for each $x\in \partial_TV$,  
\[   W(\tau)-\xi\leq  \mc H^2(\Psi_j(x))\geq W(\tau)< W(\tau)-\frac{1}{2}\xi.   \]
By the same argument as in {\bf Case I}, there exists $H:[0,2]\times \partial_TV\to \omc T^t$ such that for all 
\begin{gather*}
     \mc H^2(H(t,y))\leq \mc H^2(H(0,y))\leq W(\tau) -\xi/2 \quad \text{ for all } (t,y)\in [0,1]\times \partial_TV;\\
     \mc H^2(H(t,y))\leq W(\tau) -\frac{3}{2}\mu \quad \text{ for all } (t,y)\in [1,2]\times \partial_TV;\\
     \mf F(H(t,y),\Sigma)<\mu/4, \quad (t,y)\in [0,2]\times \partial_TV,
\end{gather*} 
and $H(2,y)=F_{w(y)}(\Sigma)$ is an embedded sphere, where $(F_v)_{v\in \oli B^{m}}$ is from White's local min-max, and $w:\partial T_V\to \partial \oli B^m$; see \cite{MN18}*{Theorem 6.1}. Observe that $H(t,y)$ is diffeomorphic to $H(0,y)$ for all $(t,y)\in[0,1]\times \partial_TV$ (c.f. the construction of $H_1$ in Theorem \ref{thm: index lower bound}), then the decreasing of the area gives
\begin{equation}\label{eq:always on boundary}
 H(0,y)\in \omc T^s\cup\omc S^t\quad \Longrightarrow\quad  H(t,y)\subset \omc T^s\cup\omc S^t\quad \text{ for all } (t,y)\in [0,1]\times \partial_TV.
 \end{equation}
Denote by $W=\partial_TV\cap \partial X$. Recall that for all $y\in \partial_TV$, $\mc H^2(\Psi_j(y))\geq W(\tau)-\xi>s$  and $\Psi_j(W)\subset \omc T^s\cup\omc S^{t}$. We conclude that $\Psi_j(W)\subset \omc S^t$. This together with \eqref{eq:always on boundary} implies that 
\begin{equation}\label{eq:[0,1]XW on St}
    H([0,1]\times W)\subset \omc S^t.
\end{equation}      
Note that $H|_{\{2\}\times \partial_T V}$ are all spheres. Recall that 
\[\mc H^2(H(t,y))\leq W(\tau) -\frac{3}{2}\mu \quad \text{ for all } (t,y)\in [1,2]\times \partial_TV.\]
Then by Lemma \ref{lem: interplation result}, we get that there exists a homotopy $H_W:[1,2]\times W\to  \omc S^{W(\tau)-5\mu/4}$ such that 
\begin{itemize}
    \item $H_W(1,y)=H(1,y)$, $H_W(2,y)=H(2,y)$;
    \item $\mf F(|H_W(t,y)|,|H(2,y)|)<2\mf F(|H(1,y)|,|H(2,y)|)<\mu/4$.
\end{itemize}

Now consider the two maps $H,H_W:[1,2]\times W\to \omc T^{W(\tau)-3\mu /4}$. By Lemma \ref{lem: interplation result} again, there exists $\wti H:[0,1]\times ([1,2]\times W)\to \omc T^{W(\tau)-\mu}$ with
\begin{itemize}
\item $\wti H(s,1,y)=H_W(1,y)=H(1,y)$,  $\wti H(s,2,y)=H_W(2,y)=H(2,y)$;
    \item $\mf F(|\wti H(s,t,y)|,H_{W}(t,y))< 2\mf F(|H_W(t,y)|,|H(2,y)|)<\mu/2$;
    \item  $\mc H^2(\wti H(s,t,y))\leq  W(\tau)-\mu/2$.
    \end{itemize}

Consider the following complex: 
\[   C_1:=V\cup ( [0,2]\times \partial_TV)\cup ([0,1]\times [1,2]\times W).\]
By identifying $\partial_TV=\{0\}\times \partial_T V$ and 
$[1,2]\times W=\{0\}\times [1,2]\times W$, 
we have that 
\[ \partial C_1=\big(\partial V\setminus \partial_TV\big)
               \cup \big([0,1]\times W\big)
               \cup \big(\{2\}\times \partial_TV\big)
               \cup \big([0,1]\times \{1,2\}\times W\big).
               \]
Define the continuous map $\psi_1:C_1\to \omc T^t$ by
\[ 
    \psi_1=\left\{
    \begin{aligned}
    &\Psi_j(x) \quad  &x\in V\\
    &H(t,y) &(t,y)\in   [0,2]\times \partial_TV\\
    &H_W(s,t,y) & (s,t,y)\in [0,1]\times [1,2]\times W.
\end{aligned} \right.
\]
\begin{claim}\label{claim:partial C1 on boundary}
    $\psi_1(\partial C_1)\subset \omc S^t$.
\end{claim}
\begin{proof}[Proof of Claim \ref{claim:partial C1 on boundary}]
By the definition of $\partial_TV$, one has that $\partial V\setminus \partial_TV\subset \partial X$. Observe that $\Psi_j(x)\geq W(\tau)-\xi>s$ for $x\in V$ and $\Psi_j(\partial V\setminus \partial_TV)\subset  \omc S^t\cup \omc T^s$. Hence 
\[  \psi_1(\partial V\setminus \partial_TV)\subset \omc S^t.\]
By \eqref{eq:[0,1]XW on St}, $\psi_1([0,1]\times W)\subset \omc S^t$. $\psi_1(\{2\}\times \partial_TV)=H(\{2\}\times \partial_TV)\subset \omc S^{W(\tau)-\mu}$ by the construction of $H$. 
For the last one, note that 
\[\psi_1([0,1]\times \{1,2\}\times W)=H_W(\{1,2\}\times W)=H(\{1,2\}\times W)\subset \omc S^{W(\tau)-\mu}\]
This finishes the proof of Claim \ref{claim:partial C1 on boundary}.
\end{proof}

Moreover, we have that $\mf F(\psi_1(x),\Sigma)<3\mu$ for all $x\in C_1$. Thus we can fill-in $\psi_1([\partial C_1])$ in $\omc S^t$ first, and then by Lemma \ref{lem: interplation result}, we conclude that $\psi_1([C_1])=0\in H_k(\omc T^t,\omc T^s\cup \omc S^t;\mb Z_2)$.

Now let  
\[  X':= (X\setminus V)\cup ( [0,2]\times \partial_TV)\cup ([0,1]\times [1,2]\times W)\]
and define $\psi_2:X'\to \omc T^t$ by
\begin{gather*}
    \psi'(x)=\Psi_j(x) \quad \text{ for } x\in X\setminus V;\\
        \psi'(x)= \psi_1(x) \quad \text{ for } x\in  [0,2]\times \partial_TV)\cup ([0,1]\times [1,2]\times W.
\end{gather*} 
Note that $\psi'_\#[X']+(\psi_1)_\#[C]=(\Psi_j)\#[X]$. Hence we conclude that 
\[  \psi'_\#[X']= (\Psi_j)\#[X]=\tau.\]
Moreover, the argument in Claim \ref{claim:partial C1 on boundary} also gives that 
\[  \mc H^2(\psi'(x))<W(\tau)-\xi/2 \quad \text{ for all } x\in  ([0,2]\times \partial_TV)\cup ([0,1]\times [1,2]\times W).\]
By the definition of $V$, we also have that 
\[  \mc H^2(\psi'(x))<W(\tau)-\xi/2\quad \text{ for all } x\in X\setminus V.\]
This contradicts the definition of $W(\tau)$. This finishes the proof of {\bf Case II}.

\medskip
\noindent{\bf Case III:} {\em Suppose that $\lim_{j\to\infty}\mathbf{L}(\Pi^j_{\partial}) = W(\tau)= s$.}

\smallskip
By Simon-Smith min-max theory together with the genus bound, index upper bound, multiplicity one theorem, there exists an embedded minimal torus or sphere $\Sigma_j$ with 
\[  \mc H^2(\Sigma_j)=\mf L(\Pi^j_{\partial}), \quad \mr{index}(\Sigma_j)\leq k-1.\]
\begin{claim}\label{claim:Sigmaj is a sphere}
    $\Sigma_j$ is a sphere for all large $j$.
\end{claim}
\begin{proof}[Proof of Claim \ref{claim:Sigmaj is a sphere}]
Suppose on the contrary that $\Sigma_j$ is a torus for all large $j$ (up to a subsequence). Observe that by definition,
\[   \Psi_j(x) \in \mc T \quad \Longrightarrow  \quad \mc \mc H^2(\Psi_j(x))<s. \]
Suppose that $x_k\in X$ with $\mc H^2(\Psi_j(x_k))\to s$. Then by the continuity of $\mc H^2(\Psi_j(\cdot))$, $\Psi_j$ converges to some element in $\omc S^s$. It follows that $\mc H^2(\Sigma_j)<s$.
By assumption, there is no embedded minimal torus with $\mr{index}\leq k-1$ and $\Area\in (s-\eps,s)$. It follows that $\mc H^2(\Sigma_j)\leq s-2\eps$. This contradicts $W(\tau)= s$.
\end{proof}
Clearly, this contradicts that there is no embedded minimal sphere with area equal to $s$. This completes the proof of Theorem \ref{thm:relative homology min-max}.

\end{proof}

\section{Morse inequalities for relative spaces}\label{sec: morse inequalities for relative spaces} 
In \cite{Marques-Montezuma-Neves20}, Marques-Montezuma-Neves established a Morse theory for minimal hypersurfaces in a Riemannian manifold by proving the strong Morse inequalities on the space of mod 2 flat chains. In this section, we follow their proof to prove that the same inequalities hold for the area functional on the space of surfaces with bounded topology in $S^3$ with positive Ricci curvature.


Given an embedded surface $\Sigma$ and $\alpha>0$, denote by $C^\infty(\Sigma,\alpha)$ the space of embedded surface which is a smooth graph over $\Sigma$ with $C^2$ graph norm less than $\alpha$. Given a set $\mc B$ of embedded surfaces, denote $C^\infty(\mc B,\alpha)=\cup_{\Sigma\in\mc B}C^\infty(\Sigma,\alpha)$. 

Given an embedded minimal surface $\Sigma$ with $\mc H^2(\Sigma) = \ell$ and $\Index(\Sigma) = k$, let $\{F_v\}_{v \in \overline{B}^k} \subset \Diff(S^3)$ be the family associated with $\Sigma$ in White's Local Min-max Theorem \cite{MN18}*{Theorem 6.1}. Then there exists $\eps_1 > 0$ such that for every $\Gamma \in C^{\infty}(\Sigma, \epsilon_1)$, $A^*(\Gamma) > A^{*}(\Sigma)$ unless $\Gamma = \Sigma$. Without loss of generality, we assume
\[
F_{v}(\Sigma) \in C^{\infty}(\Sigma, \epsilon_1/4) \quad \forall v \in \overline{B}^k. 
\]

Now write $b_j(t,s)$ the rank of $H_j(\oli{\mc T}^t,\oli{\mc T}^s\cup \oli{\mc S}^t;\mb Z_2)$.
\begin{proposition}\label{prop:local relative min-max}
Consider $g$ a metric as in Proposition \ref{prop: a choice of generic metric}. Let $r \in \mathbb{Z}_{+}$ and an embedded minimal torus $\Sigma \in \mc T_{min}(g)$ with $\Index(\Sigma) = k \leq r$ and $\mc H^2(\Sigma) = \ell$. For every $\delta > 0$, we have $\ell - \delta < s < \ell < t < \ell + \delta$ such that 
    \begin{align*}
    &b_{j}(t, s) = 0 \quad \forall j \leq r, j \neq k,\\
    &b_{k}(t, s) = 1.  
    \end{align*}
\end{proposition}
\begin{proof}
We closely follow the proof of \cite{Marques-Montezuma-Neves20}*{Proposition 3.6}, with some alterations. By assumption on $g$, $\Sigma$ is the unique $g$-minimal surface with index no more than $r$ and area lies in $(\ell - \delta', \ell + \delta')$ for some $\delta' > 0$. 

Let $s, t$ be such that $\ell - \delta' < s < \ell < t < \ell + \delta'$. To show $b_{j}(t, s) = 0$ when $j \leq r, j \neq k$, we consider a non-trivial element $\tau \in H_{k}(\oli{\mc T}^t,\oli{\mc T}^s\cup \oli{\mc S}^t; \mathbb{Z}_2)$. By the Relative Homology Min-max Theorem \ref{thm:relative homology min-max}, there exists an embedded minimal torus $\Sigma' \in \mc T_{min}(g)$ with $\Index(\Sigma') = j$ and $\mc H^2(\Sigma') \in [s, t)$, which gives a contradiction.         

We proceed to show $b_{k}(t, s) = 1$. First, using the same argument by White yields the fact:

\begin{fact}\label{fact: White} 
    For each $\zeta > 0$, there exists $\rho = \rho(\zeta, \Sigma, \eps_1) > 0$ such that if there exists an embedded torus $S$ so that $S \in C^{\infty}(\Sigma, \eps_1/2)$ and $A^{*}(S) \leq A^{*}(\Sigma) + \rho$, then $S \in C^{\infty}(\Sigma, \zeta)$.  
\end{fact}

Let $t \in (\ell, \ell +\delta)$ be such that $t - \ell < \min\{\delta', \rho(\eps_1/16, \Sigma, \eps_1)\}$. We choose $0 < \mu < \delta'/2$ small so that 
\[
\mc H^2(F_v(\Sigma)) < \mc H^2(\Sigma) - 2\mu \quad \forall v \in \partial B^k. 
\]

Consider the smooth map $P^{\Sigma}: \overline{B}^k \rightarrow \mathbb{R}^k$ defined by 
\[
P^{\Sigma}(v) = \sum_{i = 1}^k ||F_v(\Sigma)||(\eta_i) \cdot e_i. 
\]
Since $P^{\Sigma}(0) = 0$ and $DP^{\Sigma}(0) = \Id$, we have 
\[
(P^{\Sigma})_{*}[\partial \overline{B}^{k}] \neq 0 \in H_{k - 1}(\mathbb{R}^{k} \setminus \{0\}; \mathbb{Z}_2).  
\]

Define the map $\Psi': \overline{B}^k \rightarrow \oli{\mc T}^t$ by $v \mapsto F_v(\Sigma)$. Let 
\[
\tau' = \Psi'_{*}[\overline{B}^k] \in H_{k}(\oli{\mc T}^t,\oli{\mc T}^{\ell - \mu}\cup \oli{\mc S}^t; \mathbb{Z}_2)
\]
be the relative homology class naturally associated with the $k$-parameter family $\{F_v\}$. In the next two claims, we will show that $\tau'$ is the unique element which contributes to $b_k(t, s)$.     

\begin{claim}\label{claim: b_k geq 1}
    $\tau' \neq 0$. 
\end{claim}

\begin{proof}[Proof of Claim \ref{claim: b_k geq 1}]
    Suppose $\tau' = 0$. Then we can find a $(k + 1)$-dimensional $\Delta$-complex $Y$ and a continuous map $H: Y \rightarrow \oli{\mc T}^t$ so that $\partial Y$ decomposes into disjoint $k$-dimensional $\Delta$-subcomplexes 
    \[
    \partial Y = \overline{B}^k \cup Z, 
    \]
    satisfying
    \begin{enumerate}
        \item $H|_{\overline{B}^k} = \Psi'$,
        \item $H(Z) \subset \oli{\mc T}^{\ell - \mu} \cup \oli{\mc S}^t$,
        \item $H(y') \in C^{\infty}(H(y), \eps_1/16)$ for every $(k + 1)$-simplex $s \in Y$ and all $y, y' \in s$.    
    \end{enumerate}
    Now set 
    \[
    W = \cup \{(k + 1)\text{-simplices } \mk s \in Y: H(y) \in C^{\infty}(\Sigma, \eps_1/4) \text{ for some } y \in \mk s\}.
    \]
    By definition, $H(y) \in C^{\infty}(\Sigma, 3\eps_1/8)$ for every $y \in W$ and is an embedded torus. Moreover, if $\mk t$ is a $k$-simplex in $\partial W \setminus (\overline{B}^k \cup Z)$, then we have $H(y) \notin C^{\infty}(\Sigma, \eps_1/8)$ for any $y \in \mk t$. 
    
    Since $\overline{B}^{k} \subset \partial W$, we may write 
    \[
    \partial W = \overline{B}^k \cup B' 
    \]
    as a disjoint union of two $k$-dimensional $\Delta$-subcomplexes with $\partial B' = \partial \overline{B}^k$. Let $P: B' \rightarrow \mathbb{R}^k$ be a continuous map defined by 
    \[
    P(y) = \sum_{i = 1}^k ||H(y)||(\eta_i) \cdot e_i
    \]
    From $P|_{\partial B'} = P^\Sigma|_{\partial \overline{B}^k}$ we obtain
    $P_{*}([\partial B']) \neq 0 \in H_{k - 1}(\mathbb{R}^{k} \setminus \{0\}; \mathbb{Z}_2)$ and the existence of $y' \in B'$ with $P(y') = 0$ by a degree argument. Hence 
     \[
    \ell = A^{*}(\Sigma) \leq A^{*}(H(y')) < t < \ell + \rho(\eps_1/16, \Sigma, \eps_1). 
    \]
    It follows from Fact \ref{fact: White} that $H(y') \in C^{\infty}(\Sigma, \eps_1/16)$. On the other hand, since $y' \in W$, $H(y') \in C^{\infty}(\Sigma, \eps_1/2)$ and hence   
    \[
    \mc H^2(H(y')) = A^{*}(H(y')) \geq A^{*}(\Sigma) = \mc H^2(\Sigma) = \ell. 
    \]
    This together with the fact that $H(y')$ is an embedded torus implies $y' \notin Z$. Then $H(y') \notin C^{\infty}(\Sigma, \eps_1/8)$, which gives a contradiction and proves the claim. 

\end{proof}

\begin{claim}\label{claim: b_k = 1} 
    $H_{k}(\oli{\mc T}^t,\oli{\mc T}^{\ell - \mu}\cup \oli{\mc S}^t; \mathbb{Z}_2) = \{0, \tau'\}$. 
\end{claim}

\begin{proof}[Proof of Claim \ref{claim: b_k = 1}] 
    Let $\tau \neq 0 \in H_{k}(\oli{\mc T}^t,\oli{\mc T}^{\ell - \mu}\cup \oli{\mc S}^t; \mathbb{Z}_2)$. From Theorem \ref{thm:relative homology min-max} we know $W(\tau) = \mc H^2(\Sigma) = \ell$. There exists a homotopy class $\Pi$ of maps $\Psi: X \rightarrow \oli{\mc T}^t$ ($X$ is a $k$-dimensional $\Delta$-complex) relative to $\partial X$ with $\Psi(\partial X) \subset \oli{\mc T}^{\ell - \mu}\cup \oli{\mc S}^t$ such that $\mathbf{L}(\Pi) = \ell$ and $\Psi_{*}([X]) = \tau$.      

    Consider a minimizing sequence $\{\Psi_j\} \subset \Pi$. By our assumption on the metric $g$, $\Sigma$ is the only smooth element of $\mathbf{C}(\{\Psi_j\})$. Hence for $0 < h < \min\{\mu/8, h_1\}$, we obtain from Theorem \ref{thm: Deform away from nonsmooth critical varifolds} another minimizing sequence $\{\hat{\Psi}_j\} \subset \Pi$ such that $\Sigma$ is the only smooth element of $\mathbf{C}(\{\hat{\Psi}_j\})$ and for some $\xi \in (0, \mu/2)$ and $j$ large enough, 
    \[
    x\in X,\mathcal{H}^2(\hat{\Psi}_{j}(x)) \geq \ell - \xi\quad  \Longrightarrow \quad \mf F(|\hat \Psi_j(x)|,|\Sigma|)<h \text{ and } \mc F(\hat \Psi_j(x),\Sigma)<h.  
    \] 

    Now fix a large $j$ and our goal is to show $\tau = \tau'$. To this end, we interpolate between the map $\hat{\Psi}_j: X \rightarrow \oli{\mc T}^t$ and another map $\overline{\Psi}: \overline{X} \rightarrow \oli{\mc T}^t$ ($\overline{X}$ is obtained from $X$ by a cut-and-paste procedure). By construction, the non-trivial $k$-th homology part of $\overline{X}$ is related to $\overline{B}^k$ via some map $\overline{w}$ and hence 
    \[
    \tau = ({\hat{\Psi}_j})_{*}[X] = \overline{\Psi}_{*}[\overline{X}] = \Psi'_{*}[\overline{B}^k] = \tau',
    \]
    which shows Claim \ref{claim: b_k = 1}.  

    Assuming
    \[
    \sup_{x, x' \in t} \mathbf{F}(\hat{\Psi}_j(x), \hat{\Psi}_j(x')) < \min\{\xi/8, \mu/4\} \text{ for every $k$-simplex $t \in X$}
    \]
    and considering 
    \[
    V = \{k\text{-simplices } \mk t \in X: \mathcal{H}^2(\hat{\Psi}_j(x)) \geq \ell - \xi/4 \text{ for some } x \in \mk t\}.   
    \] 
    Since $\Sigma$ is the unique $g$-minimal surface with index no more than $r$ and area lies in $(\ell - \delta', \ell + \delta')$, we can deform $\hat{\Psi}_{j}|_{\partial X}$ so that $V \cap \partial X = \emptyset$. Note that
    \[
    \mathcal{H}^2(\hat{\Psi}_j(y)) \leq \ell - \xi/4 \quad \forall y \in \partial V. 
    \]
    
    By the same argument in Theorem \ref{thm: index lower bound},
    there exists $H:[0,2]\times \partial V\to \omc T^t$ (by combining $H_1$ and $H_2$) such that for all $(t,y)\in \partial V$
\begin{gather*}
     \mc H^2(H(t,y))<\ell -\xi/2; \quad  \mf F(H(t,y),\Sigma)< \mu/4, 
\end{gather*}

    Let $Q$ be the cone over $\partial V$, and let $\overline{w}: Q \rightarrow \overline{B}^k$ be a continuous extension of $w$ that sends the vertex of $Q$ to $0$, where $w: \partial V \rightarrow \partial B^k$ is a continuous function.  Consider 
    \[C = V \cup ([0, 2] \times \partial V) \cup Q\]
    and define the continuous map $\psi: C \rightarrow \oli{\mc T}$ by 
    \[  
    \psi (x)=\hat{\Psi}_j(x) \text { for } x\in V; \quad \psi (t,y)=H(t,y) \text{ for } (t,y)\in [0,2]\times \partial V; \quad \psi(q) = F_{\overline{w}(q)}(\Sigma) \text{ for } q \in Q. 
    \]
    Note that $\partial C = 0$ and $C \subset \mf B^\mf F_{\mu/4}(\Sigma)$. Hence $\psi_*[C]=0\in H_k(\omc T^t,\omc T^{\ell -\mu} \cup \omc S^t;\mb Z_2)$ by Lemma \ref{lem: interplation result} with $\mathcal{K} = \{\Sigma\}$.

    Now set $\overline{X} = (X \setminus V) \cup ([0, 2] \times \partial V) \cup Q$ and define the map $\overline{\Psi}: \overline{X} \rightarrow \oli{\mc T}^{\ell - \xi/16}$ by 
    \[
    \overline{\Psi}|_{X \setminus V} = \hat{\Psi}_{j} \quad \text{and} \quad \overline{\Psi}|_{([0, 2] \times \partial V) \cup Q} = \psi. 
    \]
    Since $(\hat{\Psi}_j)_{*}[X] = \tau$ and $\psi_{*}[C] = 0$, we have  
    \[
    \overline{\Psi}_{*}[\overline{X}] = \tau. 
    \]
    
    Consider $\Pi_{\partial}$ the homotopy class of maps $\overline{\Psi}|_{\overline{X}\setminus Q}$ relative to the boundary $\partial (\overline{X} \setminus Q)$, where $\overline{\Psi}(\partial (\overline{X} \setminus Q)) \subset \oli{\mc T}^{\ell - \mu}$. Since there is no minimal surface with index less than or equal to $k$ and area in $[\ell - \mu, \ell - \xi/16)$, we can deform $\overline{X} \setminus Q$ into the relative part $\oli{\mc T}^{\ell - \mu}\cup \oli{\mc S}^t$. In this way we obtain 
    \[
    \psi_{*}([Q]) = \tau. 
    \]
    Based on the fact that $\psi|_{Q} = \Psi' \circ \overline{w}$, we deduce the following: 
    \[
    \tau = \psi_{*}[Q] = \Psi'_{*}(\overline{w}_{*}[Q]) = \Psi'_{*}[\overline{B}^k] = \tau',   
    \]
    which completes the proof of the claim and of the proposition. 
\end{proof}\end{proof}

Note that $b_j(t,s)$ are subadditive, meaning that for $r<s<t$,
\begin{equation}\label{eq:subadd}
    b_k(t,r)\leq b_k(t,s)+b_k(s,r).
\end{equation}
Indeed, the long exact sequence induced by the short exact sequence 
\[    (\oli{\mc T}^s\cup\oli{\mc S}^t,\oli{\mc T}^r\cup\oli{\mc S}^t)\to (\oli{\mc T}^t,\oli{\mc T}^r\cup \oli{\mc S}^t)\to (\oli{\mc T}^t,\oli{\mc T}^s\cup \oli{\mc S}^t)\]
gives that 
\[
   b_k(t,r)\leq b_k(t,s)+\mr{rank}\,H_k(\omc T^s\cup\omc S^t,\omc T^r\cup\omc S^t;\mb Z_2)
\]
By Excision Theorem \cite{Hat-AT-book}*{Theorem 2.20},
\[
H_k(\omc T^s\cup\omc S^t,\omc T^r\cup\omc S^t;\mb Z_2)\simeq H_k(\omc T^s,\omc T^r\cup\omc S^s;\mb Z_2).
\]
Then \eqref{eq:subadd} follows immediately.

Denote by 
\[   \mk s_k(t,s):=b_k(t,s)-b_{k-1}(t,s)+\cdots+(-1)^kb_0(t,s).\]
Then the above long exact sequence implies that $\mk s_k(t,s)$ is also subadditive, meaning that for $r<s<t$,
\begin{equation*}
    \mk s_k(t,r)\leq \mk s_k(t,s)+\mk s_k(s,r).
\end{equation*}

Denote by $c_k(\ell)$ the number of minimal tori with index equal to $k$ and area less than $\ell$. Denote by $\beta_k(\ell)$ the rank of $H_k(\oli{\mc T}^\ell,\oli{\mc S}^\ell;\mb Z_2)$.
\begin{theorem}[Relative Strong Morse Inequalities]\label{thm:relative strong Morse inequality}
Suppose $g$ is a bumpy metric on $S^3$ with positive Ricci curvature. 
 For any fixed $\ell\in (0,\infty)$, we have $\beta_k(\ell)<\infty$ and the Strong Relative Morse Inequalities for every $k\in \mb Z_+$:
 \[    c_k(\ell)-c_{k-1}(\ell)+\cdots +(-1)^{k}c_0(\ell)\geq \beta_k(\ell)-\beta_{k-1}(\ell)+\cdots+(-1)^k\beta_0(\ell).
 \]
 In particular, we have 
 \[    c_k(\ell)\geq \beta_k(\ell)\]
for every $k\in \mb Z_+$.
\end{theorem}

\begin{proof}
We first prove the theorem for Riemannian metrics as in Proposition \ref{prop: a choice of generic metric}. Suppose that $g$ is such a metric. Fix $k\in \mb Z_+$. Let $\{\Sigma_1,\Sigma_2,\cdots,\Sigma_q\}$ be the collection of minimal tori w.r.t metric $g$ with $\mc H^2<\ell$ and $\mr{index}\leq k$. We can suppose that they are ordered so that 
\[ \mc H^2(\Sigma_q)<\mc H^2(\Sigma_{q-1}) < \cdots<\mc H^2(\Sigma_1).\]
Set $a_i=\mc H^2(\Sigma_i)$ and $m_i=\mr{index}(\Sigma_i)$. Proposition \ref{prop:local relative min-max} implies that we can choose $s_i<a_i<t_i$ for every $1\leq i\leq q$ such that $t_i<s_{i-1}$ for all $2\leq i\leq q$, $0<s_q<t_1<\ell$ and 
\[   b_j(t_i,s_i)=0 \quad \forall j\leq k, j\neq m_i; \quad b_{m_i}(t_i,s_i)=1.
\]
The Relative Homology Min-Max Theorem \ref{thm:relative homology min-max} implies 
\begin{gather*}
    b_j(\ell, t_1)=0 \quad \forall j\leq r; \quad b_j(s_q,0)=0\quad \forall j\leq k;\\
    b_j(s_{i-1},t_i)=0 \quad \forall j\leq r, 2\leq i\leq q.
\end{gather*}
Note that $\beta_k(\ell)=b_k(\ell,0)$ and $b_k(t,s)$ is subadditive. Then
\[
    \beta_k(\ell)\leq b_k(\ell, t_1)+\sum_{i=1}^q b_k(t_i,s_i)+\sum_{i=2}^qb_k(s_{i-1},t_i)+b_k(s_q,0).
    \]
The right hand side of the above inequality is equal to $c-k(\ell)$. Hence $\beta_k(\ell)<\infty$ for every $k$.

Recall that $\mk s_k(t,s)$ is also subadditive. Hence 
\[    \mk s_k(\ell,0)\leq \mk s_k(\ell, t_1)+\sum_{i=1}^q \mk s_k(t_i,s_i)+\sum_{i=2}^q\mk s_k(s_{i-1},t_i)+\mk s_k(s_q,0).\]
This together with the Relative Homology Min-Max Theorem \ref{thm:relative homology min-max} implies that 
\begin{align*}
    \mk s_k(\ell, 0)
    &\leq \sum_{i=1}^q\mk s_k(t_i,s_i)
        =\sum_{1\leq i\leq q}(-1)^{k-m_i}\\ 
    &=c_k(\ell)-c_{k-1}(\ell)+\cdots +(-1)^{k}c_0(\ell).
\end{align*} 
Therefore, 
\[   \beta_k(\ell)-\beta_{k-1}(\ell)+\cdots+(-1)^k\beta_0(\ell)
\leq c_k(\ell)-c_{k-1}(\ell)+\cdots +(-1)^{k}c_0(\ell).\]
Since $k$ is arbitrary, this proves the theorem for metrics as in Proposition \ref{prop: a choice of generic metric}. It follows immediately that the theorem is valid for bumpy metrics with positive Ricci curvature. 
\end{proof}

\section{Homology of the relative space}\label{sec: construction of relative sweepouts}
In this section, we determine the homology groups of the relative space $(\omc T,\omc S)$. Then the existence of minimal tori (for bumpy metrics with positive Ricci curvature) follows from the strong Morse inequalities; see Theorem \ref{thm:relative strong Morse inequality}. Recall that $\mc T$ is homotopic to $\mc T_{min}\simeq \mb {RP}^2\times \mb {RP}^2$, but the relative space is hard to compute directly from algebraic topology. Alternatively, we use the Relative Homology Min-max Theorem \ref{thm:relative homology min-max} and the round metric on $S^3$ to construct explicit local homology classes, and then using the strong Morse inequalities and White's Perturbation Theorem \cite{Whi91}*{Theorem 3.2} to prove that there are no other nontrivial classes. Finally, by a Lusternik-Schnirelmann type argument, we finish the proof of our Theorem \ref{thm: main theorem}. 

\subsection{A 9-parameter family of tori}

We recall the notations in \cite{MN14}.
Let $B^4\subset \mb R^4$ be the open unit ball, and $S^3=\partial B^4$ be the unit sphere. 
For each $z\in B^4$, we recall the conformal map
\[
    F_z:S^3\to S^3, \quad F_z(x)=\frac{1-|z|^2}{|x-z|^2}(x-z)-z.
\]
Consider $\Sigma\subset S^3$ a Clifford torus given by 
\[ 
    \big\{ (\frac{x}{\sqrt 2},\frac{y}{\sqrt 2})\in \mb R^2\times \mb R^2; |x|=|y|=1\big\}.
\]
Consider $\{\Sigma_t\subset S^3\}$ a one-parameter family of tori
\[  
    \big\{(\cos(\pi t/2) x,\sin (\pi t/2)y)\in \mb R^2\times \mb R^2; |x|=|y|=1\big\}.
    \]

Let us recall the family of unoriented Clifford tori in $S^3$ parametrized in \cite{ketover2022flipping}*{Section 4}. One can first identify $S^3$ and $S^2$ with the group of unit quaternions and pure unit quaternions (without real part) respectively, i.e. 
\begin{align*}
    S^3 :=\{a+b\bm i+c\bm j+d\bm k:|a|^2+|b|^2+|c|^2+|d|^2=1\}, \quad
    S^2 :=\{b\bm i+c\bm j+d\bm k\in S^3\}.
\end{align*}
As in \cite{ketover2022flipping}, for $u \in S^2$, $B\subset S^2$, denote 
\[
    \tau (u,B):=\{x\in S^3:x\in P^{-1}(u,w)\cap S^3, w\in B \},
\]
where $P: \tilde{G}_2(\mathbb R^4)\to S^2\times S^2$ is a homeomorphism. It's known that $\tau(u, B)$ is a Clifford torus if $u \in S^2$ and $B$ is a great circle of $S^2$ (cf. \cite{ketover2022flipping}*{(4.16)}). 
Given $v=(v_1,v_2,v_3)\in S^2$, denote by 
\[
    E(v):=\partial \left(\left\{x\in S^2: v_1x_1+v_2x_2+v_3x_3<0\right\}\right)\subset S^2
\]
an oriented equator. Since $\tau(u,E(v))=\tau(-u,E(-v))$ is $\tau(u,E(-v))=\tau(-u,E(v))$ with the opposite orientation, we have a family of unoriented Clifford tori in $S^3$ parametrized by $\mathcal{G}: \mathbb{RP}^2 \times \mathbb{RP}^2 \to \mc T_{min}$. 
\begin{align*}
    \mathcal{G}([u_1,u_2,u_3], [v_1,v_2,v_3]) := \spt\left(\tau\left(u_1 \bm i + u_2 \bm j + u_3 \bm k, E(v_1\bm i+v_2\bm j+v_3\bm k) \right) \right).
\end{align*}
Denote by $\mc T_{min}$ the space of Clifford tori in $S^3$ w.r.t. the round metric.  
Given an oriented Clifford torus $T$ with outward normal vector field $N$ and $t\in(-1,1)$, denote by
\[  \phi_t(T):=\{\cos(\pi t/4)x+\sin(\pi t/4) N(x); x\in T\},\]
which is a torus with signed distance (in $S^3$) $\pi t/4$ to $T$. 
Consider the space $\mb{RP}^2\times \mb {RP}^2\ttimes (-1,1)$ which is the quotient space $S^2\times S^2\times(-1,1)/\sim$ with equivalent relations 
\[    (u,v,t)\sim(-u,-v,t)\sim(-u,v,-t)\sim(u,-v,-t).\]
Then the map $(u,v,t)\mapsto\phi_t(\tau(u,E(v)))$ is a well defined map from $\mb{RP}^2\times \mb {RP}^2\ttimes (-1,1)$ to $\mc T$.
Now we have the map 
\[    \Phi:\mb{RP}^2\times \mb {RP}^2\ttimes (-1,1) \times B^4\rightarrow \mc T\] 
defined by 
\[    \Phi(u,v,t,z):=F_z(\phi_t(\tau(u,E(v)))).\]

Consider the unit normal vector field $N_t$ of $\phi_t(T)$ at $\phi_t(x)$, 
\[    N_t(\phi_t(x)) = -\sin (\pi t/4)x + \cos(\pi t/4) N(x), \quad \forall x\in T, t\in (-1,1).\] 
Observe that $\phi_{-1}(T)$ and $\phi_1(T)$ are circles and $\{\phi_t(T)\}_{-1<t<1}$ gives a foliation of $S^3\setminus \phi_{\pm 1}(T)$.
Let $\zeta:[-1,1]\to [0,1]$ be a cutoff function with 
\[   \zeta(t)=\zeta(-t); \quad \zeta(t)=1 \text{ for } |t|\leq 1/4; \quad \zeta(t)=0 \text{ for } |t|\geq 3/8.\]
Let $X(\phi_t(x))=\frac{\pi}{4}\zeta(t) N(\phi_t(x))$ be a smooth vector field defined on $S^3$. Let $(\psi_s)_{s}$ be the one parameter family of diffeomorphism of $S^3$ generated by $X$. Then by direct computation,
\[    \psi_t(T)=\phi_t(T) \quad \text{ for } |t|\leq 1/4.\]
We use $(\psi^\Sigma_s)_s$ to denote this one-parameter family $(\psi_t)$ of diffeomorphism of $S^3$ associated with $\Sigma$ and $\zeta$.

Denote by $\{\partial_i\}$ the standard orthonormal basis of $\mb R^4$.  
Now fix an oriented Clifford torus $T_0$ with unit normal vector field $N$, let 
\[    e_i:= \frac{\partial_i^\perp}{\|\partial_i^\perp\|_{L^2(\Sigma)}}.    \]
Choose $\{\eta_i\}\subset C^\infty(S^3)$ so that $|\eta_i|\leq 1$ and such that 
\[    \eta_i|_\Sigma=0, \quad \nabla \eta_i|_\Sigma=\frac{1}{4\pi^2\|\partial_i^\perp\|_{L^2(\Sigma)}}e_i.\] 
Choose $\eta_0\in C^\infty(S^3)$ so that $|\eta_0|\leq 1$ and such that $\eta_0$ is the signed distance function to $\Sigma$ in a neighborhood of $\Sigma$.
Denote by $Q:\mc T\to  \mb R^5$ as follows:
\[   Q(T)=\frac{2}{\pi^3}\|T\|(\eta_0)N+  \sum_{i=1}^4\| T\|(\eta_i) \partial_i.\]

Recall that for each oriented Clifford torus $\Sigma$, there exists a unique isomorphism $u^\Sigma$ of $S^3$ such that $u^\Sigma(T_0)=\Sigma$. Now we consider a map $Q^\Sigma: \mc T\to \mb R^5$ as follows:
\[   
    Q^\Sigma(T)=\frac{2}{\pi^3}\|(u^{\Sigma})^{-1}T\|(\eta_0)N(\Sigma) + \sum_{i=1}^4\| (u^{\Sigma})^{-1}T\|(\eta_i) \partial_i,\]
where $N(\Sigma)$ is the unit normal vector of $\Sigma$.

Consider the map $P:\mc T_{min}\ttimes (-1,1)\times B^4\to \mc T_{min}\ttimes\mb R\times \mb R^4$ given by
\[   P(\Sigma, t,z)=\Big(\Sigma,Q^\Sigma(F_z(\Sigma_t))\Big).\]
Then one can verify that $P(\Sigma,0,0)=(\Sigma,0)$, and for any Clifford torus $\Sigma$, $DP(\Sigma,\cdot,\cdot)=id$. We have the following facts of Clifford torus and one can verify it easily 
\begin{enumerate}
    \item \[   \frac{d}{ds}F_{s\partial_i}|_{s=0}(z)=2(\partial_i-\langle \partial_i, z\rangle  ).\]
    \item every Clifford torus has index 5 and the Jacobi operator has  eigenfunctions are 1, $\langle N,\partial_i \rangle$, $i=1,2,3,4$ w.r.t. the eigenvalues $4, 2,2,2,2$.
\end{enumerate}

Given an embedded surface $\Sigma$ and $\alpha>0$, denote by $C^\infty(\Sigma,\alpha)$ the space of embedded surface which is a smooth graph over $\Sigma$ with $C^2$ graph norm less than $\alpha$. Given a set $\mc B$ of embedded surfaces, denote $C^\infty(\mc B,\alpha)=\cup_{\Sigma\in\mc B}C^\infty(\Sigma,\alpha)$.

Given a Clifford torus and a constant $c_0>0$, consider the functional on $\mc T$ given by 
\[    A^*_\Sigma (T) :=\|T\|(S^3)+c_0 \|Q^\Sigma(T)\|^2.\]
Then one can check that for some large constant $c_0>0$, $\Sigma$ is a strictly stable critical point for $A^*_\Sigma$. Moreover, the same argument by White also gives that there exists $\gamma_1>0$ such that for every embedded torus $T\in C^\infty(\Sigma,\gamma_1)$, we have $A^*_\Sigma(T)>A^*_\Sigma(T)$ unless $T=\Sigma$. This implies that for each $\alpha>0$, there exists $\kappa=\kappa(\alpha,\gamma_1)>0$ such that if an embedded torus $T\in C^\infty(\Sigma,\delta_1)$ and $A^*_\Sigma(T)\leq 2\pi^2+\kappa$, then $T\in C^\infty(\Sigma,\alpha)$.

\begin{lemma}\label{lem:5-family constants}
Given $\delta>0$, there exists $\gamma>0$ such that 
\begin{enumerate}
    \item for all $u,v,t,z$ with $|t|\leq \delta, |z|\leq \delta$, 
    \[    \mc H^2(\Phi(u,v,t,z))\leq 2\pi^2;\]
    \item for all $u,v,t,z$ with $|t|+|z|\geq  \delta/2$, 
    \[    \mc H^2(\Phi(u,v,t,z))< 2\pi^2-\gamma;\]
\end{enumerate}
Moreover, given $\eps_1>0$, there exists $\delta>0$ such that $\Phi(u,v,t,z)$ is a graph over $\Phi(u,v,0,0)$ with $C^2$ norm less than or equal to $\eps_1$ for $|t|\leq \delta, |z|\leq \delta$.
\end{lemma}

\subsection{Non-trivial local relative homology classes}
For simplicity, denote by $\mc X_9(r):=\mb{RP}^2\times \mb {RP}^2\ttimes [-r,r] \times \oli B^4(r)$ for $r\in(0,1)$.

For any $0<s<2\pi^2<t$, We define the relative homology class
\[  \oli \sigma:=\Phi_*[\mc X_9(\delta)]\in H_9(\omc T^{t},\omc T^s\cup \omc S^t;\mb Z_2),\]
where $[\mc X_9(\delta)]$ denotes the generator of $H_9(\mc X_9(\delta),\partial \mc X_9(\delta);\mb Z_2)$.
\begin{lemma}\label{lem:nontrivial homology}
For every $\eps>0$, there exists $2\pi^2-\eps<s<2\pi^2<t<2\pi^2+\eps$ such that the relative homology class $\oli\sigma\neq 0$.
\end{lemma}
\begin{proof}
Let $\gamma_1$ be the constant as above. Choose $\eps_1\in (0,\gamma_1/3)$. Then by Lemma \ref{lem:5-family constants}, there exists $\delta>0$ such that $\Phi(u,v,t,z)\in C^\infty(\Phi(u,v,0,0),\eps_1)$ for all $|t|\leq \delta, |z|\leq \delta$. Furthermore, there exists $\gamma>0$ such that for $|t|+|z|\geq \delta/2$,
\[    \mc H^2(\Phi(u,v,t,z))<2\pi^2-\gamma.\]

Choose $t\in (2\pi^2,2\pi^2+\eps)$ with $t-2\pi^2<\kappa(\eps_1,\gamma_1)$ and $s\in (2\pi^2-\eps,2\pi^2)$ with $2\pi^2-s<\gamma$.

Suppose $\oli\sigma=0$ by contradiction. This implies that there exists a 10-dimensional $\Delta$-complex $Y$ and a map $H:Y\to \omc T^{t}$, continuous in the smooth topology, such that $\partial Y$ is the union of two 9-dimensional $\Delta$ subcomplexes 
\[    \partial Y=\mc X_9(\delta)\cup Z,\]
disjoint as subcomplexes, such that
\begin{gather*}
H|_{\mc X_9(\delta)}=\Phi,\\
\sup_{y\in Z} \mf M(H(y))<s.
\end{gather*}

Note that by the choice of $\eps_1$ and $\delta$, $H(x)\in C^\infty(\mc T_{min},\eps_1)$ for $x\in \mc X_9(\delta)$.
Now let $Y'\subset Y$ be the compact set such that for each $x\in Y'$, $H(x)\in C^\infty(\mc T_{min},2\eps_1)$. Clearly, $\mc X_9(\delta)\subset Y'$.
We can apply the barycentric subdivision to $Y$ (still denoted by $Y$) so that for each 10-dimensional $\mk s$ with $\mk s\cap Y'\neq \emptyset$, then for all $x\in \mk s$, $H(x)\in C^\infty(\mc T_{min},3\eps_1)$. 
Let $W$ be the union of all 10-simplices $\mk s\in Y$ that intersects $Y'$. By definition, for all $x\in Y$, $H(x)\in C^\infty(\mc T_{min},3\eps_1)$. 
In particular, $H(x)$ is a smooth torus for all $x\in W$.

Consider a 9-simplex $\mk t\in \partial W$ such that $\mk t\notin \mc X_9(\delta)\cup Z$. Then there exists $\mk s\in Y$ with $\mk s\notin W$ and $\mk t\subset \partial \mk s$. By definition of $W$, $\mk s$ does not intersect $Y'$, meaning that for each $x\in \mk t$, $H(x)\notin C^\infty(\mc T_{min},2\varepsilon_1)$.

We can decompose $\partial W$ as the union of two 9-dimensional subcomplexes, disjoint as subcomplexes:
\[    \partial W=\mc X_9(\delta)\cup B'.\]
Note that $\partial B'=\partial \mc X_9(\delta)$. Recall that by Smale Conjecture, there exists a deformation retract $\mc D:\mc T\to \mc T_{min}$. 
Consider the continuous map $P:B'\to \mc T_{min}\ttimes \mb R\times \mb R^4$ by 
\[    \wti P(y):= \Big(\mc D(H(y)),Q^{\mc D(H(y))}(H(y))\Big)  .\]
Note that $\wti P|_{\partial B'}=\wti P|_{\partial \mc X_9(\delta)}=  P|_{\partial \mc X_9(\delta)}$. Therefore, 
\[    \wti P_*([\partial B'])\neq 0\in H_8(\mc T_{min} \ttimes \mb R\times \mb R^4\setminus \mc T_{min}\times \{(0,0)\};\mb Z_2).
\]This implies the existence of $y_0\in B'$ such that $\wti P(y_0)=(\mc D(\Phi(y)),0,0)$. So 
\[    A^*_{\mc D(\Phi(y))}(H(y_0))=\|H(y_0)\|(M)<t. \]
This implies $A^*_{\mc D(\Phi(y))}(H(y_0))<t<2\pi^2+\kappa(\eps_1,\gamma_1)$. Since $y_0\in W$, then 
\[H(y_0)\in C^\infty(\mc T_{min},3\eps_1)\subset C^\infty(\mc T_{min},\gamma_1),\]
and hence $H(y_0)\in C^\infty(\mc T_{min},\eps_1)$.

On the other hand, $\mf M(H(y_0))=A^*_{\mc D(H(y))}(H(y_0))\geq 2\pi^2$. This gives that $y_0\notin Z$. Hence $H(y_0)\notin C^\infty(\mc T_{min},2\eps_1)$. Contradiction, thus $\oli\sigma\neq 0$. This completes the proof of Lemma \ref{lem:nontrivial homology}.
\end{proof}

Note that by Poincar\'{e} duality \cite{Hat-AT-book}*{Theorem 3.43}, the cap product with a fundamental class $[M]\in H_9(\mc X_9(\delta), \partial 
 \mc X_9(\delta); \mb Z_2)$ gives isomorphisms $D:H^k(\mc X_9(\delta);\mb Z_2)\to H_{9-k}(\mc X_9(\delta),\partial \mc X_9(\delta);\mb Z_2)$. Note that for each closed submanifold $\Gamma\subset \mb {RP}^2\times \mb{RP}^2$ which represents a nontrivial homology class $[\Gamma]\in H_j(\mb{RP}^2\times \mb{RP}^2;\mb Z_2)$, we have that 
 \[    [\Gamma\ttimes[-\delta,\delta]\times \oli B^4]\neq 0\in H_{j+5}(\mc X_{9}(\delta),\partial \mc X_9(\delta);\mb Z_2).\]

 We conclude that 
 \[
 H_j(\mc X_{9}(\delta),\partial \mc X_9(\delta);\mb Z_2)=\left\{\begin{aligned}
     &0  \quad &i=1,2,3,4\\
     &\mathbb{Z}_2 = \langle [[-\delta,\delta]\times \oli B^4(\delta)] \rangle  & i = 5\\
     &\mathbb{Z}_2 \oplus \mathbb{Z}_2   & i = 6\\
     &\mathbb{Z}_2 \oplus \mathbb{Z}_2 \oplus \mathbb{Z}_2  &i = 7\\
     &\mathbb{Z}_2 \oplus \mathbb{Z}_2 &i = 8\\
     &\mathbb{Z}_2 = \langle [\mathbb{RP}^2\times \mb{RP}^2 \ttimes[-\delta,\delta] \times \oli B^4(\delta) ]\rangle & i = 9
 \end{aligned}
 \right.
 \]
 
Then the same argument in Lemma \ref{lem:nontrivial homology} gives that 
for each closed submanifold $\Gamma\subset \mb {RP}^2\times \mb{RP}^2$ which represents a nontrivial homology class $[\Gamma]\in H_j(\mb{RP}^2\times \mb{RP}^2;\mb Z_2)$, $\Phi_*[\Gamma\ttimes[-\delta,\delta]\times \oli B^4]\neq 0\in H_{j+5}(\omc T^t,\omc T^s\cup\omc S^t))$ by taking suitable $s<2\pi^2<t$.  
\begin{lemma}\label{lem:local nontrivial homology injective}
For every $\eps>0$, there exists $2\pi^2-\eps<s<2\pi^2<t<2\pi^2+\eps$ such that the relative homology class $\Phi_*:H_*(\mc X_9(\delta),\partial \mc X_9(\delta);\mb Z_2 )\to H_*(\omc T^t,\omc T^s\cup\omc S^t)$ is injective.
\end{lemma}

\subsection{Homology groups of relative spaces}
Now we are ready to determine the homology groups $H_*(\omc T,\omc S;\mb Z_2)$.
\begin{proposition}\label{prop:homology isomorphism}
There exists an isomorphism $\omc D:H_j(\mb{RP}^2\times\mb{RP}^2;\mb Z_2)\to H_{j+5}(\omc T,\omc S;\mb Z_2)$ for all integers $j\geq 0$.
\end{proposition}
\begin{proof}
Let $t,s$ be the two real numbers in Lemma \ref{lem:local nontrivial homology injective}. Recall that by  proved by S. Brendle \cite{Brendle-Lawson-conj}, every embedded minimal torus in the round $S^3$ is a Clifford torus.

\begin{claim}\label{claim:trvial toplogy for larger than t}
    $H_j(\omc T,\omc T^t\cup\omc S;\mb Z_2)$ is trivial for all $j\geq 1$.
\end{claim} 
\begin{proof}[Proof of Claim \ref{claim:trvial toplogy for larger than t}]
Suppose not, there exists a relative homology group $[\tau_j]\in H_j (\omc T,\omc S;\mb Z_2)$. Let $g_k$ be a sequence of bumpy metrics on $S^3$ smoothly converging to the round metric $g_{round}$.
Then by the Relative Homology Min-max Theorem \ref{thm:relative homology min-max}, for each $k$, there exists an embedded minimal torus $\Sigma_k$ with index $k$ and 
\[    \mc H^2_{g_k}(\Sigma_k)= W_{g_k}([\tau_j]):= \inf_{[\sum \mk s_i] = \tau} \sup_{i, x \in \Delta^k} \mathcal{H}^2_{g_k}(\mk s_i(x)). \]
Note that as $k\to\infty$,
\[    W_{g_k}([\tau_j])\to W_{g_{round}}([\tau_j])\geq t.\]
Thus by the compactness of embedded minimal surfaces, $\Sigma_k$ converges to an embedded minimal torus in $S^3$ w.r.t. the round metric, which must be a Clifford torus. Recall that the Clifford torus has area $2\pi^2<t$, which leads to a contradiction. Hence the claim is proved.
\end{proof}

Similarly, the same argument gives that 
\begin{claim}\label{claim:trvial toplogy for less than s}
    $H_j(\omc T^s,\omc S^s;\mb Z_2)$ is trivial for all $j\geq 1$.
\end{claim}   

Observe that by Excision Theorem \cite{Hat-AT-book}*{Theorem 2.20} and Claim \ref{claim:trvial toplogy for less than s},
\[  H_j(\omc T^s\cup\omc S^t,\omc S^t;\mb Z_2)\simeq H_j(\omc T^s,\omc S^s;\mb Z_2)=\{0\} \quad \text{ for all } j\geq 1.\] 
Now consider the short exact sequence 
\[    (\omc T^s\cup \omc S^t,\omc S^t)\to (\omc T^t,\omc S^t)\to (\omc T^t,\omc T^s\cup \omc S^t).\]
The induced long exact sequence gives that for all $j\geq 1$,
\begin{equation}\label{eq:t0 to ts}
  H_j(\omc T^t,\omc S^t;\mb Z_2)\simeq H_j(\omc T^t,\omc T^s\cup \omc S^t;\mb Z_2).
  \end{equation} 
Then consider the short exact sequence 
\[   (\omc T^t\cup \omc S,\omc S)\to (\omc T,\omc S)\to (\omc T,\omc T^t\cup \omc S).\] 
The induced long exact sequence together with Claim \ref{claim:trvial toplogy for larger than t} gives that 
\[    H_j(\omc T^t\cup \omc S,\omc S;\mb Z_2)\simeq H_j(\omc T,\omc S;\mb Z_2).\]
Then by Excision Theorem \cite{Hat-AT-book}*{Theorem 2.20} and \eqref{eq:t0 to ts}, 
\[    H_j(\omc T,\omc S;\mb Z_2)\simeq H_j(\omc T^t\cup \omc S,\omc S;\mb Z_2)\simeq H_j(\omc T^t,\omc S^t;\mb Z_2)\simeq H_j(\omc T^t,\omc T^s\cup \omc S^t;\mb Z_2).\]
This together with Lemma \ref{lem:local nontrivial homology injective} implies that 
\[   \mr{rank}\,H_j(\omc T,\omc S;\mb Z_2)\geq \mr{rank}\, H_{j+5}(\mb {RP}^2\times \mb{RP}^2).\]
In particular,
\begin{equation}\label{eq:rank larger than 9}
     \sum_{j=1}^\infty\mr{rank}\,H_j(\omc T,\omc S;\mb Z_2)\geq \sum_{j=1}^\infty\mr{rank}\, H_{j+5}(\mb {RP}^2\times \mb{RP}^2)=9
\end{equation}

It remains to prove the equality. Recall that by White \cite{Whi91}*{Theorem 3.2}, there exists a sequence of metrics $\wti g_k\to g_{round}$ such that each $(S^3,g_k)$ has exactly 9 embedded minimal tori with index less than or equal to 9. By perturbing $g_k$ (still denoted by $g_k$), one can assume that $g_k$ is bumpy. Then by the strong Morse inequalities,
\[    \sum_{j}\mr{rank}\,H_j(\omc T,\omc S;\mb Z_2)\leq 9.\]
This together with \eqref{eq:rank larger than 9} gives the equality.    
\end{proof}

\subsection{Cohomology and Lusternik-Schnirelmann Theory}
For simplicity, denote by $\mc X:=\mc X_9(\delta)$.
Denote by $\lambda,\mu$ the two generators of $H^1(\mc X_9(\delta);\mb Z_2)\simeq H^1(\mb {RP}^2\times \mb {RP}^2)$ such that 
\[   \lambda(\mb {RP}^1_2\times \{q\})= \mu(\{p\}\times \mb {RP}^1_1)=0 ;\quad  \lambda(\mb {RP}^1_1\times \{q\})= \mu(\{p\}\times \mb {RP}^1_2)=1.\]
Denote by $\bar\lambda$ the generator of $\mc Z_2(S^3;\mb Z_2)$.  
Consider the natural inclusion 
\[  
\begin{tikzcd}
   \Psi: \mc X\arrow[r, "\Phi"] &\mc T\arrow[r, "\iota_1"] &\omc T\arrow[r, "\iota_2"] &\mc Z_2(S^3;\mb Z_2).
\end{tikzcd}
\]
Since $\Psi|_{\mb {RP}^1_1}$ and $\Psi|_{\mb {RP}^1_2}$ are both nontrivial sweepouts, then 
\[   \bar\lambda\Big( \Psi_*([\mb {RP}^1_1])\Big)= \bar\lambda \Big(\Psi_*([\mb {RP}^1_1]\Big)=1. \]
This implies that 
\[ \Psi^*(\bar\lambda)=\lambda +\mu.\]   
Consider the cap product 
\[ \begin{tikzcd}[column sep=small]
H_{k+5}(\mc X,\partial \mc X;\mb Z_2)\arrow[d] &\times &H^1(\mc X;\mb Z_2) \arrow [r, "\smallfrown"]&  H_{k+4}(\mc X,\partial \mc X;\mb Z_2)\arrow[d,"\simeq"]\\
H_{k+5}(\omc T^t,\omc T^s\cup \omc S^t;\mb Z_2)&\times &H^1(\omc T^t;\mb Z_2) \arrow[u]\arrow [r, "\smf"]&  H_{k+4}(\omc T^t,\omc T^s\cup \omc S^t;\mb Z_2)
\end{tikzcd}\]
is natural ($k=1,2,3,4$). Observe that the homeomorphism 
\[   \mc C:H_{k+3}(\mc X,\partial \mc X;\mb Z_2)\to  H_{k}(\mc X,\partial \mc X;\mb Z_2)\quad     \sigma\mapsto \sigma\smf (\lambda+\mu)^3 
\]
is a nontrivial for $k=5,6$. Together with the fact that 
\[H^1(\mc Z_2(S^3;\mb Z_2);\mb Z_2)\to H^1(\omc T^t;\mb Z_2)\to H^1(\mc X;\mb Z_2)
\]
is surjective, hence we have that 
\[   \mc C:H_{k+3}(\omc T^t,\omc T^s\cup \omc S^t;\mb Z_2)\to  H_{k}(\omc T^t,\omc T^s\cup \omc S^t;\mb Z_2)  \quad \sigma\mapsto \sigma\smf u^3\]
is nontrivial for some element $u\in H^1(\omc T^t;\mb Z_2)$. Note that by the naturality of the long exact sequences in Proposition \ref{prop:homology isomorphism}, we conclude that
\[ 
 \mc C:H_{k+3}(\omc T,\omc S;\mb Z_2)\to  H_{k}(\omc T,\omc S;\mb Z_2)  \quad \sigma\mapsto \sigma\smf u^ 3 \quad k=5,6
\]
is nontrivial. 

Recall that for any $\tau\in H_*(\omc T,\omc S;\mb Z_2)$,
\[ W(\tau) = \inf_{[\sum \mk s_i] = \tau} \sup_{i, x \in \Delta^k} \mathcal{H}^2(\mk s_i(x)). \]

\begin{theorem}
Let $g$ be a metric on $S^3$ with positive Ricci curvature.  Suppose that $(S^3, g)$ contains only finitely many embedded minimal tori. Then 
\[
    0 < W(\sigma\smf u^3)<W(\sigma\smf u^2)<W(\sigma \smf u)<W(\sigma),
   \]
where $\sigma\in H_9(\omc T,\omc S;\mb Z_2)$ is the generator. 
As a consequence, any $(S^3,g)$ with positive Ricci curvature contains at least four distinct embedded minimal tori.    
\end{theorem}

\begin{proof}
The proof is essentially the same as the Lusternik-Schnirelmann argument in \cite{MN17}*{Theorem 6.1} (see also \cite{wangzc2023existenceFour}*{Appendix E}). We provide more details here for the sake of completeness.

    We prove the last inequality and the others are similar. Suppose by contradiction that $W(\sigma) = W(\sigma \smf u)$. Let $\Phi: (X, \partial X) \rightarrow (\oli{\mc T}, \oli{\mc S})$ be a $9$-sweepout such that 
    \[
    \Phi_{*}[X] = \sigma\in H_9(\omc T,\omc S;\mb Z_2). 
    \]
    Let $\mc T_{min}$ be the collection of integral varifolds, with mass equal to $W(\sigma)$, whose support are embedded minimal tori. Note that $\mc T_{min}$ is a finite set by assumption.

    Set $Y = \{x \in X: \mathbf{F}(|\Phi(x)|, \mc T_{min}) \geq \eta\}$ and $Z = \overline{X \setminus Y}$. We have $\partial X$ contained in $Y$ but disjoint from $Z$. Let $i_1: Z \rightarrow X$ and $i_2: Y \rightarrow X$ be two natural inclusion maps. Our goal is show that $\Phi \circ i_2: (Y, \partial X) \rightarrow (\oli{\mc T}, \oli{\mc S})$ satisfies
    \[
    (\Phi \circ i_2)_{*}[Y] = \sigma \smf u\in H_8(\omc T,\omc S;\mb Z_2). 
    \]
    Once this is true, the Relative Homology Min-max Theorem \ref{thm:relative homology min-max} gives an embedded minimal torus $\Sigma$ with $\mathcal{H}^2(\Sigma) = W(\sigma \smf u) = W(\sigma)$. Hence $\Sigma \in \mc T_{min}$ and we derive a contradiction. 

    To fulfill our goal, consider the following diagram  
    \[ \begin{tikzcd} [column sep=tiny]
    H_{9}(\oli{\mc T}, \oli{\mc S};\mb Z_2) &\times &H^1(\oli{\mc T};\mb Z_2) \arrow [r, "\smallfrown"] \arrow[d, "\Phi^*"]&  H_{8}(\oli{\mc T},\oli{\mc S};\mb Z_2)\\
    H_{9}(X, \partial X;\mb Z_2) \arrow[u, "\Phi_{*}"] \arrow[d, "r"] &\times &H^1(X;\mb Z_2) \arrow [r, "\smf"] &  H_{8}(X, \partial X;\mb Z_2)\arrow[u,"\Phi_*"] \arrow[d, "j_2"]\\
    H_{9}(X, X;\mb Z_2) = 0 &\times &H^1(X, Z;\mb Z_2) \arrow[u, "j_1"] \arrow [r, "\smf"]&  H_{8}(X, Y;\mb Z_2)
\end{tikzcd}\]
where the maps $j_{1}$ and $j_{2}$ come from the following long exact sequences: 
\[  
\begin{tikzcd}
    \cdots \arrow[r] & H_8(Y, \partial X; \mb Z_2) \arrow[r, "(i_2)_{*}"] & H_8(X, \partial X; \mb Z_2) \arrow[r, "j_2"] & \mc H_{8}(X, Y; \mb Z_2) \arrow[r] & \cdots.  
\end{tikzcd}
\]
\[  
\begin{tikzcd}
    \cdots \arrow[r] & H^1(X, Z; \mb Z_2) \arrow[r, "j_1"] & H^1(X; \mb Z_2) \arrow[r, "{i_1}^{*}"] & \mc H^1(Z; \mb Z_2) \arrow[r] & \cdots.  
\end{tikzcd}
\]

By naturality of cap product, the diagram is commutative. We already know $\sigma \smf u \neq 0 \in H_{8}(\oli{\mc T},\oli{\mc S};\mb Z_2)$ and $[X] \neq 0 \in H_{9}(X, \partial X;\mb Z_2)$. Then there exists $\phi \neq 0 \in H^1(X)$ so that $\Phi^{*}(\phi) = u$. It follows that $[X] \smf \phi \neq 0 \in H_{8}(X, \partial X;\mb Z_2)$ with $\Phi_{*}([X] \smf \phi) = \sigma \smf u$. 

Since $(\Phi \circ i_1)(Z) \subset \mathbf{B}^{\mathbf{F}}_{\eta}(\mc T_{min})$, we have $(\Phi \circ i_1)^* = {i_1}^{*} \circ \Phi^* = 0$. Hence $0 = (i_{1}^{*} \circ \Phi^{*})(u) = i_{1}^*(\phi)$. By exactness, there exists $\xi \neq 0 \in H^1(X, Z; \mb Z_2)$ so that $j_1(\xi) = \phi$.  

Now, suppose by contradiction that $(\Phi \circ i_2)_{*}[Y] \neq \sigma \smf u$. We claim that $j_2([X] \smf \phi) \neq 0 \in H_{8}(X, Y; \mb Z_2)$. It follows from $(\Phi \circ i_2)_{*}[Y] \neq \sigma \smf u$ that $[X] \smf \phi \neq (i_2)_{*}[Y]$. Then the first long exact sequence implies $j_2([X] \smf \phi) \neq 0$.   

As the diagram is commutative, we have 
\[
0 \neq j_2([X] \smf \phi) = r([X]) \smf \xi = 0.
\]
This gives a contradiction and completes the whole proof. 
\end{proof}

\appendix

\bibliographystyle{amsalpha}
\bibliography{minmax}
\end{document}